\documentclass{amsart} 

\usepackage[all]{xy}
\usepackage{amssymb,latexsym}
\usepackage{mltex}
\usepackage{amsmath}
\usepackage[colorlinks=true,linkcolor=blue,citecolor=blue]{hyperref}
 \usepackage{lscape}
\usepackage{enumerate}
\usepackage{amsthm}
\usepackage{hyperref}
\usepackage{multicol}

\newtheorem{thm}{Theorem}
\newtheorem{cor}{Corollary}
\newtheorem{lem}{Lemma}[section]
\newtheorem{prop}[lem]{Proposition}
\newtheorem{rem}{Remark}
\newtheorem{defn}[lem]{Definition}

\newcommand{\C}{\mathbb{C}}

\newcommand{\HH}{\mathbb{H}}
\newcommand{\HC}{\mathbb{H}^{\C}}
\newcommand{\R}{\mathbb{R}}

\newcommand{\Z}{\mathbb{Z}}

\def\bfz{F}

\DeclareMathOperator{\re}{Re}
\DeclareMathOperator{\im}{Im}

\begin{document}
\title{Constant angle surfaces in 4-dimensional Minkowski space}

\author{Pierre Bayard, Juan Monterde, Ra\'ul C. Volpe}

\maketitle

\begin{abstract}
We first define a complex angle between two oriented spacelike planes in  4-dimensional Minkowski space, and then study the constant angle surfaces in that space, i.e. the oriented spacelike surfaces whose tangent planes form a constant complex angle with respect to a fixed spacelike plane. This notion is the natural Lorentzian analogue of the notion of constant angle surfaces in 4-dimensional Euclidean space. We prove that these surfaces have vanishing Gauss and normal curvatures, obtain representation formulas for the constant angle surfaces with regular Gauss maps and construct constant angle surfaces using PDE's methods. We then describe their invariants of second order and show that a surface with regular Gauss map and constant angle $\psi\neq 0\ [\pi/2]$ is never complete. We finally study the special cases of surfaces with constant angle $\pi/2\ [\pi],$ with real or pure imaginary constant angle and describe the constant angle surfaces in hyperspheres and lightcones. 
\end{abstract}

\date{}
\maketitle\pagenumbering{arabic}
\section*{Introduction}
A constant angle surface in $\R^3$ is a surface whose tangent planes have a constant angle with respect to some fixed direction in $\R^3.$ Constant angle surfaces in $\R^3$ have been studied in \cite{DRH,RH}. This notion has then been extended to other geometric contexts, especially to hypersurfaces in $\R^n$ \cite{DRH}, surfaces in $\R^4$ \cite{BdSOR} or surfaces in 3-dimensional Minkowski space $\R^{1,2}$ \cite{LM}. The aim of the paper is to introduce the notion of constant angle surfaces in 4-dimensional Minkowski space $\R^{1,3},$ and study their main properties. We will first observe that a natural complex angle is defined between two oriented spacelike planes in $\R^{1,3},$ and then define a constant angle surface in $\R^{1,3}$ as an oriented spacelike surface whose tangent planes form a constant complex angle with respect to some fixed oriented spacelike plane. This notion appears to be the natural Lorentzian analogue of the notion of constant angle surface in $\R^4,$ and also extends the definitions in the literature of constant angle surfaces in $\R^3$ and in $\R^{1,2}.$ Let us note that the definition of the complex angle between two oriented spacelike planes seems to be new and might be of independent interest. We will suppose throughout the paper that the surfaces have regular Gauss map: under this additional assumption, we will obtain general representation formulas for the constant angle surfaces in $\R^{1,3}.$ These formulas rely on a representation formula for surfaces with regular Gauss map and vanishing Gauss and normal curvatures in $\R^{1,3}$ given in \cite{GMM}, and reformulated in terms of spin geometry in \cite{Bay}. With these formulas at hand, the construction of a constant angle surface amounts to solving a PDE system. We will then describe the invariants of second order of the constant angle surfaces and prove that a surface with regular Gauss map and constant angle $\psi\neq 0\ [\pi/2]$ is never complete. Finally, we will show that a surface has constant angle $\pi/2\ [\pi]$ if and only if it is a product of curves in orthogonal planes, we will study the surfaces with real or pure imaginary constant angle and describe the constant angle surfaces in hyperspheres and lightcones. Additionally to the constructions of the constant angle surfaces using PDE's methods, we will give explicit concrete examples of constant angle surfaces in the spirit of \cite{MV}.

The paper is organized as follows. We describe the Clifford algebra and the spin group of $\R^{1,3}$ in Section \ref{section clifford}, we introduce the complex angle between two oriented spacelike planes of $\R^{1,3}$ in Section \ref{section complex angle} and the notion of constant angle surface in $\R^{1,3}$ in Section \ref{section first properties}. We then give general representation formulas and reduce the construction of these surfaces to the resolution of a PDE in Sections \ref{section representation 1} and \ref{section representation 2}, and introduce a frame adapted to a constant angle surface and describe the second order invariants of a constant angle surface in Section \ref{section adapted frame}. We study the completeness of the constant angle surfaces in Section \ref{section complete}. We finally study the constant angle surfaces which are product of plane curves in Section \ref{section product}, surfaces with real or pure imaginary constant angle in Section \ref{section real im angle} and constant angle surfaces in hyperspheres and lightcones in Section \ref{section spheres lightcones}. Three appendices end the paper: we give in the first appendix an alternative construction and an elementary characterization of the complex angle between two oriented spacelike planes, we describe in the second appendix the very special case of the surfaces of constant angle $0\ [\pi]$ and we detail in the third appendix the computations leading to an explicit frame adapted to a constant angle surface.
\section{Clifford algebra and spin group of $\R^{1,3}$}\label{section clifford}
The Minkowski space $\R^{1,3}$ is the space $\R^4$ endowed with the metric
$$g=-dx_0^2+dx_1^2+dx_2^2+dx_3^2.$$
We recall here the description of the Clifford algebra and the spin group of $\R^{1,3}$ using the complex quaternions, as introduced in \cite{Bay}, and refer to \cite{Fr} for the general basic properties of the Clifford algebras and the spin groups. Let $\HH^{\C}$ be the space of complex quaternions, defined by
$$\HH^{\C}:=\{q_01+q_1I+q_2J+q_3K,\ q_0,q_1,q_2,q_3\in\C\}$$
where $I,J,K$ are such that
$$I^2=J^2=K^2=-1,\hspace{1cm} IJ=-JI=K.$$
Using the Clifford map 
\begin{eqnarray}
\R^{1,3}&\rightarrow&\HC(2)\label{Clifford map}\\
(x_0,x_1,x_2,x_3)&\mapsto& \left(\begin{array}{cc}0&ix_01+x_1I+x_2J+x_3K\\-ix_01+x_1I+x_2J+x_3K&0\end{array}\right)\nonumber
\end{eqnarray}
where $\HC(2)$ stands for the set of $2\times 2$ matrices with entries belonging to the set of complex quaternions $\HC,$ the Clifford algebra of $\R^{1,3}$ is
$$Cl(1,3)=\left\{\left(\begin{array}{cc}a&b\\\widehat{b}&\widehat{a}\end{array}\right),\ a,b\in\HC\right\}$$
where, if $\xi=q_01+q_1I+q_2J+q_3K$ belongs to $\HC,$ we denote  
$$\widehat{\xi}:=\overline{q_0}1+\overline{q_1}I+\overline{q_2}J+\overline{q_3}K,$$
the element of $\HC$ obtained by the usual complex conjugation of its coefficients.
The Clifford sub-algebra of elements of even degree is
\begin{equation}\label{identification Cl0}
Cl_0(1,3)=\left\{\left(\begin{array}{cc}a&0\\0&\widehat{a}\end{array}\right),\ a\in\HC\right\}\simeq\HC.
\end{equation}
Let us consider the bilinear map $H:\HC\times\HC\rightarrow\C$ defined by
\begin{equation}\label{H HC}
H(\xi,\xi')=q_0q_0'+q_1q_1'+q_2q_2'+q_3q_3'
\end{equation}
where $\xi=q_01+q_1I+q_2J+q_3K$ and $\xi'=q_0'1+q_1'I+q_2'J+q_3'K.$ It is $\C$-bilinear for the natural complex structure $i$ on $\HC.$ We consider
\begin{equation}\label{spin13 S3}
Spin(1,3):=\{q\in\HC:\ H(q,q)=1\}\hspace{.5cm}\subset \hspace{.5cm}Cl_0(1,3).
\end{equation}
Using the identification 
\begin{eqnarray*}
\R^{1,3}&\simeq&\{ix_01+x_1I+x_2J+x_3K,\ (x_0,x_1,x_2,x_3)\in\R^4\}\\
&\simeq&\{\xi\in\HC:\ \xi=-\widehat{\overline{\xi}}\},
\end{eqnarray*}
 where, if $\xi=q_01+q_1I+q_2J+q_3K$ belongs to $\HC,$ we denote  $\overline{\xi}=q_01-q_1I-q_2J-q_3K,$ we get the double cover
\begin{eqnarray}
\Phi:Spin(1,3)&\stackrel{2:1}{\longrightarrow}& SO(1,3)\label{double cover}\\
q&\mapsto &(\xi\in\R^{1,3}\mapsto q\ \xi\ \widehat{q}^{-1}\in\R^{1,3}).\nonumber
\end{eqnarray}
Here and below $SO(1,3)$ stands for the component of the identity of the group of Lorentz transformations of $\R^{1,3}.$ 

\section{Complex angle between two oriented spacelike planes}\label{section complex angle}
We introduce here the complex angle between two oriented spacelike planes in $\R^{1,3}.$ The definition relies on a model of the Grassmannian of the oriented spacelike planes in $\R^{1,3}$ introduced in \cite{Bay}: it is naturally a complex 2-sphere in $\C^3,$ and the intuition in Euclidean space $\R^3$ will then easily lead to the definition. Throughout the paper we will assume that $\R^{1,3}$ is oriented by its canonical basis $(e_0^o,e_1^o,e_2^o,e_3^o),$ and that it is also oriented in time by $e_0^o$: we will say that a timelike vector is \emph{future-oriented} if its first component in the canonical basis is positive.
\subsection{The Grassmannian of the oriented spacelike 2-planes in $\R^{1,3}$}\label{section model grassm}
We consider $\Lambda^2\R^{1,3},$ the vector space of bivectors of $\R^{1,3}$ endowed with its natural metric $\langle.,.\rangle$ (which has signature (3,3)). The Grassmannian of the oriented spacelike 2-planes in $\R^{1,3}$ identifies with the submanifold of unit and simple bivectors
$$\mathcal{Q}=\{\eta\in\Lambda^2\R^{1,3}:\ \langle \eta,\eta\rangle=1,\ \eta\wedge\eta=0\}.$$
The Hodge $*$ operator $\Lambda^2\R^{1,3}\rightarrow \Lambda^2\R^{1,3}$ is defined by the relation
\begin{equation}\label{i lambda2}
\langle *\eta,\eta'\rangle=\eta\wedge\eta'
\end{equation}
for all $\eta,\eta'\in \Lambda^2\R^{1,3},$ where we identify $\Lambda^4\R^{1,3}$ to $\R$ using the canonical volume element $e_0^o\wedge e_1^o\wedge e_2^o\wedge e_3^o$ on $\R^{1,3}.$ It satisfies $*^2=-id_{\Lambda^2\R^{1,3}}$ and thus $i:=-*$ defines a complex structure on $\Lambda^2\R^{1,3}.$ We also define
\begin{equation}\label{H lambda2}
H(\eta,\eta')=\langle \eta,\eta'\rangle-i\ \eta\wedge\eta'\hspace{.5cm}\in\hspace{.5cm}\C
\end{equation}
for all $\eta,\eta'\in\ \Lambda^2\R^{1,3}.$ This is a $\C$-bilinear map on $\Lambda^2\R^{1,3},$ and we have
$$\mathcal{Q}=\{\eta\in\Lambda^2\R^{1,3}:\ H(\eta,\eta)=1\}.$$
The bivectors
$$E_1=e_2^o\wedge e_3^o,\ E_2=e_3^o\wedge e_1^o,\ E_3=e_1^o\wedge e_2^o$$
form a basis of $\Lambda^2\R^{1,3}$ as a vector space over $\C;$ this basis is such that $H(E_i,E_j)=\delta_{ij}$ for all $i,j.$ Since the Clifford map (\ref{Clifford map}) identifies the vectors of $\R^{1,3}$ with some elements of $\HC(2),$ we can use this map and the usual matrix product to identify $\Lambda^2\R^{1,3}$ with elements of $\HC(2).$ Recalling (\ref{identification Cl0}), these even elements may in turn be identified with elements of $\HC$. With these identifications, we easily get
$$E_1=I,\ E_2=J,\ E_3=K$$
and
$$\Lambda^2\R^{1,3}=\{Z_1I+Z_2J+Z_3K\in\HC:(Z_1,Z_2,Z_3)\in\C^3\};$$
moreover, the complex structure $i$ and the quadratic map $H$ defined above on $\Lambda^2\R^{1,3}$ coincide with the natural complex structure $i$ and the quadratic map $H$ defined on $\HC,$ and
\begin{equation}\label{Q S2}
\mathcal{Q}=\{Z_1I+Z_2J+Z_3K:\ Z_1^2+Z_2^2+Z_3^2=1\}\ =\ Spin(1,3)\ \cap\ \Im m\ \HC,
\end{equation}
where $\Im m\ \HC$ stands for the linear space generated by $I,$ $J$ and $K$ in $\HC.$ Let us finally note that $Spin(1,3)$ and $\mathcal{Q}$ identify respectively to the complex spheres 
$$\mathbb{S}^3_{\C}:=\{Z_01+Z_1I+Z_2J+Z_3K:\ Z_0^2+Z_1^2+Z_2^2+Z_3^2=1\}$$ 
and 
$$\mathbb{S}^2_{\C}:=\{Z_1I+Z_2J+Z_3K:\ Z_1^2+Z_2^2+Z_3^2=1\}$$ 
(by (\ref{spin13 S3}) and (\ref{Q S2})); we will consider below the bundle
\begin{eqnarray}
\mathbb{S}^3_{\C}&\rightarrow& \mathbb{S}^2_{\C}\label{bundle S3 S2}\\
q&\mapsto& q^{-1}Iq.\nonumber
\end{eqnarray}
This is a principal bundle of group 
$$\mathbb{S}^1_{\C}=\{\cos A\ 1+\sin A\ I,\ A\in\C\}\ \subset \mathbb{S}^3_{\C}$$
acting on $\mathbb{S}^3_{\C}$ by multiplication on the left; it is equipped with a natural horizontal distribution: a complex curve $g:\C\rightarrow\ \mathbb{S}^3_{\C}$ will be said to be \textit{horizontal} if $g'g^{-1}$ belongs to $\C J\oplus\C K.$

\subsection{Definition of the complex angle}\label{section def angle}
\begin{defn}\label{def complex angle}
Let $p,q\in\mathcal{Q}$ be two oriented spacelike planes of $\R^{1,3}.$ The complex angle between $p$ and $q$ is the complex number $\psi\in\C$ such that
$$H(p,q)=\cos\psi.$$
It is uniquely defined up to sign and the addition of $2\pi\Z$ since $\cos:\C\rightarrow\C$ is surjective and $\cos\psi=\cos\psi'$ if and only if $\psi=\pm\psi'+2k\pi,$ $k\in\Z.$
\end{defn}
The complex angle between $p$ and $q\in\mathcal{Q}$ is a kind of complex arc-length between these two points in the complex sphere $\mathcal{Q}=S^2_{\C}:$
\begin{prop}\label{prop q p psi}
Let $p,q\in\mathcal{Q}$ be two oriented spacelike planes of $\R^{1,3}$ and $\psi\in\C$ be the complex angle between $p$ and $q.$ The following holds:
\\
\\a- If $\psi\neq 0\ [\pi],$ there exists $V\in T_{p}\mathcal{Q}$ such that $H(V,V)=1$ and
\begin{equation}\label{q p v}
q=\cos\psi\ p+\sin\psi\ V.
\end{equation}
b-  If $\psi=0\ [\pi],$ there exists $\xi\in T_{p}\mathcal{Q}$ such that $H(\xi,\xi)=0$ and
\begin{equation}\label{q p xi}
q=\pm p+\xi
\end{equation}
where the sign is positive if $\psi=0\ [2\pi]$ and negative if $\psi=\pi\ [2\pi].$
\end{prop}
\begin{proof}
Since $H(p,q)=\cos\psi$ and $H(p,p)=1,$  $\xi:=q-\cos\psi\ p$ is such that
$$H(p,\xi)=H(p,q)-\cos\psi\ H(p,p)=0$$
i.e. $\xi$ belongs to 
$$T_{p}\mathcal{Q}=\{\xi\in \mathcal{I}m\ \HC:\ H(p,\xi)=0\}.$$
There are two cases:
\\\textit{Case 1:} $H(\xi,\xi)\neq 0,$ which is equivalent to 
$$H(q,q)+\cos^2\psi\ H(p,p)-2\cos\psi\ H(p,q)\neq 0,$$
which reads $\cos^2\psi\neq 1$ since $H(p,p)=H(q,q)=1$ and $H(p,q)=\cos\psi,$ i.e. $\psi\neq 0\ [\pi];$ if $\gamma\in\C$ is such that $\gamma^2=H(\xi,\xi)$ then $\gamma\neq 0$ and setting $V=\xi/\gamma$ we obtain $H(V,V)=1$ and $q=\cos\psi\ p+\gamma\ V.$ Since $H(p,p)=H(q,q)=1$ and $H(p,V)=0$ we deduce that $1=\cos^2\psi+\gamma^2,$ i.e. $\gamma=\pm\sin\psi.$ Changing $V$ by $-V$ if necessary, we may suppose that $\gamma=\sin\psi,$ and (\ref{q p v}) holds.
\\\textit{Case 2:} $H(\xi,\xi)=0,$ which is equivalent to $\psi=0\ [\pi];$ we thus obtain $\xi=q-\epsilon p$ with $\epsilon=1$ if $\psi=0\ [2\pi]$ and $\epsilon=-1$ if $\psi=\pi\ [2\pi],$ and (\ref{q p xi}) follows.
\end{proof}
\begin{rem}\label{rem psi degenerate hyperplane}
The case $\psi=0\ [\pi]$ has the following geometric meaning: by (\ref{q p xi}), this means that $q=\pm p+\xi$ for some $\xi\in T_p\mathcal{Q}$ with $H(\xi,\xi)=0,$ which holds if and only if $p$ and $q$ belong to some degenerate hyperplane of $\R^{1,3}:$ by (\ref{H lambda2}), $\xi\in T_p\mathcal{Q}$ and $H(\xi,\xi)=0$ mean that $\xi=u\wedge N$ for a unit vector $u\in p$ and a null vector $N$ normal to $p,$ and the degenerate hyperplane containing $p$ and $q$ is $p\oplus\R N.$ Note that there are exactly two degenerate hyperplanes containing a given spacelike 2-plane $p:$ the hyperplanes $p\oplus L$ and $p\oplus L'$ where $L$ and $L'$ are the two null lines in $\R^{1,3}$ which are normal to $p.$ 
\end{rem}
\subsection{Interpretation of the complex angle in terms of two real angles}
Let us consider two oriented spacelike planes $p,q\in\mathcal{Q}$ and the complex angle $\psi$ between $p$ and $q,$ and assume that $\psi\neq 0\ [\pi].$ By Proposition \ref{prop q p psi}, there exists $V\in T_{p}\mathcal{Q}$ such that $H(V,V)=1$ and
$$q=\cos\psi\ p+\sin\psi\ V.$$
By (\ref{H lambda2}) the conditions $V\in T_p\mathcal{Q}$ and $H(V,V)=1$ mean that $V$ is of the form $u\wedge v$ for some unit spacelike vectors $u$ and $v$ belonging to $p$ and $p^{\perp},$ i.e. such that $p=u\wedge u^{\perp}$ and  $p^{\perp}=v\wedge v^{\perp}$ (where $p^{\perp}$ is the timelike plane orthogonal to $p,$ with its natural orientation, and $u^{\perp}$ and $v^{\perp}$ are respectively unit spacelike and unit timelike vectors). Writing $\psi=\psi_1+i\psi_2$ and using the formulas
\begin{eqnarray*}
\cos\psi&=&\cos\psi_1\cosh\psi_2-i\sin\psi_1\sinh\psi_2\\
\sin\psi&=&\sin\psi_1\cosh\psi_2+i\cos\psi_1\sinh\psi_2,
\end{eqnarray*}
a direct computation yields
\begin{eqnarray}
q&=&\cos\psi\ u\wedge u^{\perp}+\sin\psi\ u\wedge v\nonumber\\
&=&(\cosh\psi_2\ u+\sinh\psi_2\ v^{\perp})\wedge (\cos\psi_1\ u^{\perp}-\sin\psi_1\ v);\label{formule q uv}
\end{eqnarray}
the plane $q$ is generated by the unit and orthogonal vectors $\cos\psi_1\ u^{\perp}-\sin\psi_1\ v$ and $\cosh\psi_2\ u+\sinh\psi_2\ v^{\perp};$ these vectors are determined by an euclidean rotation of angle $\psi_1$  in the spacelike plane generated by $u^{\perp}$ and $v,$ and by a lorentzian rotation of angle $\psi_2$ in the timelike plane generated by $u$ and $v^{\perp}.$
\begin{rem}\label{rem real imaginary angles}
We deduce that the angle $\psi\neq 0\ [\pi]$ between the planes $p$ and $q$ is a real number (i.e. $\psi_2=0$) if and only if there exists a spacelike hyperplane containing both $p$ and $q$ (by (\ref{formule q uv}) this is the hyperplane $p\oplus\R v$) and that $\psi$ is a pure imaginary complex number (i.e. $\psi_1=0$) if and only there exists a timelike hyperplane containing $p$ and $q$ (this is the hyperplane $p\oplus\R v^{\perp}$). See also Proposition \ref{prop pos p p0} in Appendix \ref{app angle}.
\end{rem}
\section{First properties of a surface with constant angle}\label{section first properties}
\subsection{The Gauss map of a spacelike surface} \label{subsection gauss map}
Let us consider an oriented spacelike surface $M$ in $\R^{1.3}.$ We identify the oriented Gauss map of $M$ with the map
$$G:\ M\rightarrow \mathcal{Q},\ x\mapsto G(x)=u_1\wedge u_2,$$
where $(u_1,u_2)$ is a positively oriented orthonormal basis of $T_xM.$ We define the \textit{vectorial product} of two vectors $\xi,\xi'\in\Im m\ \HC$ by
$$\xi \times \xi':=\frac{1}{2}\left(\xi\xi'-\xi'\xi\right)\ \in\ \Im m\ \HC.$$ 
We also define the \textit{mixed product} of three vectors $\xi,\xi',\xi''\in \Im m\ \HC$ by
$$[\xi,\xi',\xi'']:= H(\xi \times \xi',\xi'')\ \in\ \C.$$
The mixed product is a \textit{complex volume form} on $\Im m\ \HC$ (i.e. with complex values, $\C$-linear and skew-symmetric with respect to the three arguments); it induces a natural \textit{complex area form} $\omega_{\mathcal{Q}}$ on $\mathcal{Q}$ by
$${\omega_{\mathcal{Q}}}_p(\xi,\xi'):=[\xi,\xi',p]$$ 
for all $p\in\mathcal{Q}$ and all $\xi,\xi'\in T_p\mathcal{Q}.$ Note that ${\omega_{\mathcal{Q}}}_p(\xi,\xi')=0$ if and only if $\xi$ and $\xi'$ are linearly dependent over $\C.$ We now recall the following expression for the pull-back by the Gauss map of the area form $\omega_Q:$
\begin{prop}\cite{Bay}\label{prop pull back} 
If $K$ and $K_N$ denote the Gauss and the normal curvatures of $M$ in $\R^{1,3},$ we have
\begin{equation}\label{formula pull-back omega}
G^*\omega_{\mathcal{Q}}=(K+iK_N)\ \omega_M,
\end{equation}
where $\omega_M$ is the form of area of $M.$ In particular, $K=K_N=0$ at $x_o\in M$ if and only if the linear space $dG_{x_o}(T_{x_o}M)$ belongs to some complex line in $T_{G(x_o)}\mathcal{Q}.$
\end{prop}
As a consequence of the proposition, if $K=K_N=0$ and if $G:M\rightarrow\mathcal{Q}$ is a regular map (i.e. if $dG_x$ is injective at every point $x$ of $M$), there is a unique complex structure $\mathcal{J}$ on $M$ such that
$$dG_x(\mathcal{J}u)=i\ dG_x(u)$$
for all $x\in M$ and all $u\in T_xM.$ Indeed, for all $x\in M,$ ${G^*\omega_{\mathcal{Q}}}_x=\omega_{\mathcal{Q}}(dG_x,dG_x)=0$ in that case, and $Im(dG_x)$ is a complex line in $T_{G(x)}\mathcal{Q}$ and we may set 
$$\mathcal{J}u:={dG_x}^{-1}(i\ dG_x(u))$$
for all $u\in T_xM.$ The complex structure $\mathcal{J}$ coincides with the complex structure introduced in \cite{GMM}. Note that $M$ cannot be compact under these hypotheses, since, on the Riemann surface $(M,\mathcal{J}),$ the Gauss map $G=G_1I+G_2J+G_3K$ is globally defined, non-constant, and such that  $G_1,G_2$ and $G_3$ are holomorphic functions. Thus, assuming moreover that $M$ is simply connected, by the uniformization theorem $(M,\mathcal{J})$ is conformal to an open set of $\C,$ and thus admits a globally defined conformal parameter $z=x+iy.$ 
\subsection{Definition of a constant angle surface}
\begin{defn}
An oriented spacelike surface $M$ is of constant angle with respect to a spacelike plane $p_o$ if the angle function $\psi$ between $p_o$ and the tangent planes of $M$ is constant. 
\end{defn}
By Proposition \ref{prop q p psi} and Remark \ref{rem psi degenerate hyperplane}, if $\psi=0\ [\pi]$ then the surface $M,$ if it is connected, belongs to a degenerate hyperplane $p_o\oplus L$ where $L$ is one of the two null lines which are normal to $p_o,$ and nothing else can be said since reciprocally an arbitrary surface in $p_o\oplus L$ has constant angle $\psi=0\ [\pi].$ Details are given in the Appendix \ref{app case psi 0 mod pi}. Assuming thus that $\psi\neq 0\ [\pi],$ this alternatively means that the Gauss map image of $M$ belongs to the complex circle of center $p_o$ and constant radius $\cos\psi$ in $\mathcal{Q},$ i.e. $G$ is of the form
\begin{equation}\label{general circle}
G=\cos\psi\ p_o+\sin\psi\ V
\end{equation}
for some function $V:M\rightarrow T_{p_o}\mathcal{Q}$ such that $H(V,V)=1$ (Proposition \ref{prop q p psi}). Since the Gauss map image of a constant angle surface in $\R^{1,3}$ is thus by definition a complex curve, it is clear from (\ref{formula pull-back omega}) that such a surface has vanishing Gauss and normal curvatures $K=K_N=0.$ It is known that constant angles surfaces in $\R^4$ have the same property \cite{BdSOR}.
\subsection{First examples: constant angle surfaces in a hyperplane}
If $M$ is a constant angle surface of $\R^{1,3},$ of angle $\psi\in\C,$ it is clear from Remark \ref{rem real imaginary angles} that if $M$ belongs to a spacelike (resp. timelike) hyperplane of $\R^{1,3}$ then $\psi$ is a real (resp. pure imaginary) number. Since constant angle surfaces in $\R^3$ and in $\R^{1,2}$ were studied in \cite{DRH,RH} and \cite{LM}, we will be merely interested in the following in surfaces which do not belong to spacelike or timelike hyperplanes. We will see below examples of constant angle surfaces belonging to degenerate hyperplanes of $\R^{1,3}$ (the angle is $\psi=0\ [\pi]$), and, in contrast with the case $\psi=0\ [\pi],$ examples of real or pure imaginary constant angle surfaces which do not belong to any hyperplanes of $\R^{1,3}.$
\subsection{A new example}\label{section new example}
Let us verify that the immersion
\begin{equation}\label{F new example}
F(x,y) = e^{ax-by}\left( \cosh(x),\sinh(x),\cos(y),\sin(y) \right),\hspace{.5cm} (x,y)\in\R^2
\end{equation}
defines a spacelike surface with constant angle. The associated tangent basis is
\[
\begin{array}{rl}
\partial_xF &= e^{ax-by}\left( a \cosh(x)+\sinh(x),\cosh(x)+a\sinh(x),a\cos(y),a\sin(y) \right) ,
\\[12pt]
\partial_yF &= -e^{ax-by}\left( b \cosh(x),b\sinh(x),b\cos(y)+\sin(y),-\cos(y)+b\sin(y) \right)
\end{array}
\]
and the first fundamental form is
$$e^{2(a x-by)}(dx^2+dy^2).$$
The immersion is thus spacelike, and $z=x+iy$ is a conformal parameter. By a direct computation, its Gauss map reads
\begin{eqnarray*}
G(z)&:=&\frac{\partial_xF\wedge \partial_yF}{|\partial_xF\wedge \partial_yF|}\\
&=&(a+ib)E_1-(\cosh z+(a+ib)\sinh z)E_2+i((a+ib)\cosh z+\sinh z)E_3.
\end{eqnarray*}
Thus, the complex angle $\psi$ between the tangent plane and the plane $p_o:=E_1$ satisfies
$$\cos\psi:=H(G(z),E_1)=a+ib;$$
it is constant. Note that $|F|^2=0,$ which means that the surface belongs to the lightcone at $0.$ We may also verify by a direct computation that the mean curvature vector $\vec{H}$ of the immersion is lightlike i.e. satisfies $|\vec{H}|^2=0$. Moreover, it is not difficult to see that the surface belongs to a hyperplane if and only if $a=\pm 1$ and $b=0,$ i.e. $\psi=0\ [\pi];$ in that case it belongs in fact to the degenerate hyperplane $x_0\pm x_1=1.$ This in accordance with Remark \ref{rem psi degenerate hyperplane}. Writing $\psi=\psi_1+i\psi_2,$ $\psi_1,\psi_2\in\R,$ and since
$$a=\cos\psi_1\cosh\psi_2\hspace{.5cm}\mbox{and}\hspace{.5cm}b=-\sin\psi_1\sinh\psi_2,$$
we obtain for $a\in(-1,1)$ (resp. $a\in \R\backslash[-1,1]$) and $b=0$ a surface with real (resp. pure imaginary) constant angle which does not belong to any hyperplane. Let us finally note that this example shows that there exist constant angle surfaces for arbitrary values of the angle $\psi.$ We will explain below how to systematically construct constant angle surfaces in $\R^{1,3}.$ 

\section{Representation of a surface with constant angle}\label{section representation 1}
In the sequel we will consider a complex circle in $\mathcal{Q}=\mathbb{S}^2_{\C},$ with center $I$ and radius $\cos\psi,$ $\psi\neq 0\ [\pi],$ and parametrized by
\begin{equation}\label{G function z}
G(z)=\cos\psi\ I+\sin\psi\ J\left\{\cos\left(\frac{2z}{\sin\psi}\right)+\sin \left(\frac{2z}{\sin\psi}\right)I\right\},\ z\in \C.
\end{equation}
It is of the form (\ref{general circle}) and such that $H(G'(z),G'(z))=4$ for all $z\in\C.$ Recall the principal bundle $S^3_{\C}\rightarrow S^2_{\C}$ introduced in (\ref{bundle S3 S2}) and its natural horizontal distribution. Direct computations show the following:
\begin{lem}\label{lift G}
The function $g:\C\rightarrow \mathbb{S}^3_{\C}$ defined by 
\begin{eqnarray}
g(z)&=&-\cos\frac{\psi}{2}\ \sin\left(z\tan\frac{\psi}{2}\right)1+\cos\frac{\psi}{2}\ \cos\left(z\tan\frac{\psi}{2}\right)I\label{explicit g}\\
&&+\sin\frac{\psi}{2}\ \cos\left(z\cot\frac{\psi}{2}\right)J-\sin\frac{\psi}{2}\ \sin\left(z\cot\frac{\psi}{2}\right)K\nonumber
\end{eqnarray}
is an horizontal lift of the function $G$ defined in (\ref{G function z}); it is such that $H(g',g')=1$ and satisfies 
\begin{equation}\label{relation g beta}
g'g^{-1}=\cos\beta J +\sin\beta K
\end{equation}
with
\begin{equation}\label{explicit beta}
\beta=-2z \cot\psi.
\end{equation}
Moreover, an horizontal lift of $G$ is necessarily of the form $g_a:=ag$ for some constant $a=\cos(A)1+\sin(A)I,$ $A\in\C,$ and satisfies
$$g_{a}'g_{a}^{-1}=\cos\beta_a J +\sin\beta_a K$$
with $\beta_a=\beta+2A.$
\end{lem}
We now write the representation theorem for the flat surfaces with flat normal bundle and regular Gauss map in $\R^{1,3}$ (Corollary 5 in \cite{Bay}, after \cite{GMM}) in the context of the prescribed Gauss map in the form (\ref{G function z}). In the statement, the function $g:\C\rightarrow\mathbb{S}^3_{\C}$ is an horizontal lift of $G$ and $\beta:\C\rightarrow\C$ is such that (\ref{relation g beta})-(\ref{explicit beta}) hold.
\begin{thm}\label{thm representation}
Let $\mathcal{U}\subset\C$ be a simply connected open set. If $h_0,h_1:\mathcal{U}\rightarrow\R$ are two real functions such that the vector fields $\alpha_1,\alpha_2\in\Gamma(T\mathcal{U})$ defined by
\begin{eqnarray*}
\alpha_1&:=&ih_0\cos\beta+h_1\sin\beta\\
\alpha_2&:=&ih_0\sin\beta-h_1\cos\beta
\end{eqnarray*}
are linearly independent at every point of $\mathcal{U}$ and satisfy $[\alpha_1,\alpha_2]=0$ (as real vector fields), then, setting 
\begin{equation}\label{xi thm}
\xi:=g^{-1} (\omega_1J+\omega_2K)\ \widehat{g}
\end{equation}
where $\omega_1,\omega_2:T\mathcal{U}\rightarrow\R$ are the dual forms of $\alpha_1,\alpha_2\in\Gamma(T\mathcal{U}),$ 
$$F=\int\xi:\ \mathcal{U}\rightarrow\R^{1,3}$$ 
is a spacelike surface with constant angle $\psi.$ Reciprocally, up to a rigid motion of $\R^{1,3},$ a spacelike surface of constant angle $\psi$ and regular Gauss map may be locally written in that form. 
\end{thm}
We also recall from \cite{Bay} that the real functions $h_0$ and $h_1$ in the theorem are the components of the mean curvature vector of the surface in a parallel frame normal to the surface, and the complex functions $\alpha_1,\alpha_2$ are the expressions in $z$ of a parallel frame tangent to the surface; moreover, these parallel frames are positively oriented, in space and in time.
\begin{rem}
Suppose that $g:\C\rightarrow \mathbb{S}^3_{\C}$ is an horizontal lift of $G:\C\rightarrow\mathbb{S}^2_{\C}$ as in Lemma \ref{lift G} and that $\beta:\C\rightarrow\C$ is such that (\ref{relation g beta}) holds.  For $a\in \mathbb{S}^3_{\C},$ the function $g_a:=ga:\C\rightarrow \mathbb{S}^3_{\C}$ also satisfies (\ref{relation g beta}), with the same function $\beta$. If $h_0$ and $h_1$ are real functions as in the statement of Theorem \ref{thm representation}, the corresponding forms $\xi_g$ and $\xi_{g_a}$ are linked by $\xi_{g_a}=a^{-1}\xi_g\widehat{a}:$ the immersions $\int \xi_g$ and $\int \xi_{g_a}$ are thus congruent, i.e. differ one from the other by a rigid motion of $\R^{1,3}$ (recall (\ref{double cover})).
\end{rem}

\section{The representation theorem in terms of the metric}\label{section representation 2}

We aim to apply the representation theorem to construct all the constant angle surfaces and give general explicit expressions in some special cases. In order to do this, we reformulate here the representation theorem (Theorem \ref{thm representation}) using the coefficients of the metric instead of the unknown functions $h_0,h_1$: the compatibility condition on these functions will then reduce to a hyperbolic PDE on the metric coefficients, whose Cauchy problem is solvable.

\subsection{Determination of the metric}
Let us keep the notation of the previous section and assume that the hypotheses of Theorem \ref{thm representation} hold. Writing $\beta=u+iv,$ straightforward computations yield
$$\alpha_1=(h_0 \sinh(v)+h_1 \cosh(v))\sin u+i(h_0 \cosh(v)+h_1 \sinh(v))\cos u$$
and 
$$\alpha_2=-(h_0 \sinh(v)+h_1 \cosh(v))\cos u+i(h_0 \cosh(v)+h_1 \sinh(v))\sin u;$$
since $\alpha_1$ and $\alpha_2$ are everywhere independent vectors in $\R^2$ we have
$$(h_0 \sinh(v)+h_1 \cosh(v))(h_0 \cosh(v)+h_1 \sinh(v))\neq 0,$$
and we may set $\mu$ and $\nu$ such that 
\begin{equation}\label{mu nu h0 h1}
\frac{1}{\mu}=h_0 \sinh(v)+h_1 \cosh(v)\hspace{.5cm}\mbox{and}\hspace{.5cm} \frac{1}{\nu}=h_0 \cosh(v)+h_1 \sinh(v)
\end{equation}
and get the formulas
\begin{equation} \label{a1 a2 tangent frame}
\alpha_1=\frac{1}{\mu}\sin(u)+\frac{i}{\nu}\cos(u)\hspace{.5cm}\mbox{and}\hspace{.5cm}\alpha_2=-\frac{1}{\mu}\cos(u)+\frac{i}{\nu}\sin(u).
\end{equation}
Since the tangent frame $(\alpha_1,\alpha_2)$ is supposed to be orthonormal (by (\ref{xi thm}) the metric is $\omega_1^2+\omega_2^2$), the metric reads 
\begin{equation}\label{metric x y}
\mu^2 dx^2+\nu^2 dy^2.
\end{equation} 
We write in the next lemma the condition $[\alpha_1,\alpha_2]=0$ appearing in Theorem \ref{thm representation} in terms of the metric coefficients $\mu,\nu:$

\begin{lem} \label{c1 c2 lemma}
The condition $[\alpha_1,\alpha_2]=0$ reads
\begin{equation} \label{Eq::Condition}
\left\lbrace
\begin{array}{rl}
\dfrac{1}{\nu}\ \partial_y \mu&=-c_1 \\[12pt]
\dfrac{1}{\mu}\ \partial_x \nu&=c_2
\end{array}
\right.
\end{equation}
with
\begin{equation}\label{c1 c2 psi1 psi2}
c_1= -\frac{\sin (2 \psi_1)}{\sin^2(\psi_1)+\sinh^2(\psi_2)}\hspace{.5cm}\mbox{and}\hspace{.5cm} c_2= -\frac{\sinh (2 \psi_2)}{\sin^2(\psi_1)+\sinh^2(\psi_2)}.
\end{equation}
\end{lem}
\begin{proof}
A straightforward computation yields
\[
 [\alpha_1,\alpha_2]  =  \left( \frac{\partial_y\mu+\nu\ \partial_xu}{\mu^2 \nu}\right) \partial_x  + \left( \frac{-\partial_x\nu+\mu\ \partial_yu}{\mu \nu^2}\right) \partial_y,
\]
and  $[\alpha_1,\alpha_2]=0$ if and only if 
\begin{equation}\label{mu nu ux uy}
\partial_y\mu=-\nu\ \partial_xu\hspace{.5cm} \mbox{and}\hspace{.5cm} \partial_x\nu=\mu\ \partial_yu.
\end{equation} 
We have by definition 
\[
u=\re (\beta)=-2\re (z \cot (\psi)),
\]
which implies
\begin{equation}\label{ux uy psi}
\partial_xu = -2\re (\cot (\psi))\hspace{.5cm}\mbox{and}\hspace{.5cm} \partial_yu = 2\im (\cot (\psi)). 
\end{equation}
Writing $\psi=\psi_1+i\psi_2$, we easily get $\partial_xu=c_1$ and $\partial_yu=c_2$ where $c_1$ and $c_2$ are given by (\ref{c1 c2 psi1 psi2}), and obtain from (\ref{mu nu ux uy}) the system (\ref{Eq::Condition}). 
\end{proof}
\begin{rem}
Computing the Christoffel symbols of the metric (\ref{metric x y}) and setting
$$T_1:=\frac{1}{\mu}\partial_x,\hspace{1cm}T_2:=\frac{1}{\nu}\partial_y$$
it appears that (\ref{Eq::Condition}) is equivalent to the equations
\begin{equation}\label{nabla Ti}
\nabla T_1=(c_1dx+c_2dy)\ T_2\hspace{.5cm}\mbox{and}\hspace{.5cm}\nabla T_2=-(c_1dx+c_2dy)\ T_1.
\end{equation}
\end{rem}
\subsection{Reformulation of the representation theorem}
We may then reformulate Theorem \ref{thm representation} as follows:
\begin{thm}\label{thm representation 2}
Let us assume that $G:\C\rightarrow S^2_{\C}$ is given by (\ref{G function z}), $g:\C\rightarrow S^3_{\C}$ is an horizontal lift of $G$ and $\beta=u+iv:\C\rightarrow\C$ is such that (\ref{relation g beta}) holds. If $\mu$ and $\nu$ are non-vanishing solutions of (\ref{Eq::Condition}) on a simply connected open set $\mathcal{U}\subset\C,$ then, setting
\begin{equation}\label{omegas u mu nu}
\omega_1= \sin (u) \mu\ dx+  \cos (u) \nu\ dy,\hspace{1cm}\omega_2=-\cos (u) \mu\ dx+  \sin (u) \nu\ dy
\end{equation}
and
\begin{equation}\label{xi g omegas}
\xi:=g^{-1} (\omega_1J+\omega_2K)\ \widehat{g},
\end{equation}
the formula
$$F=\int\xi:\ \mathcal{U}\rightarrow\R^{1,3}$$ 
defines a spacelike surface with constant angle $\psi=\psi_1+i\psi_2.$ Moreover, the metric is $\mu^2dx^2+\nu^2dy^2.$ Reciprocally, up to a rigid motion of $\R^{1,3},$ a spacelike surface of constant angle $\psi$ and regular Gauss map may be locally written in that form. 
\end{thm}
\begin{proof}
In order to show that this is a reformulation of Theorem \ref{thm representation}, we only observe that $\omega_1$ and $\omega_2$ are the dual forms of two independent vectors fields $\alpha_1$ and $\alpha_2$ $\in\Gamma(T\mathcal{U})$ such that $[\alpha_1,\alpha_2]=0:$ the forms $\omega_1$ and $\omega_2$ are independent since $\mu\nu\neq 0$ at every point, and their dual vectors fields $\alpha_1$ and $\alpha_2$ are given by \eqref{a1 a2 tangent frame}; moreover, by  Lemma \ref{c1 c2 lemma} they are such that $[\alpha_1,\alpha_2]=0$ since $\mu$ and $\nu$ are solutions of (\ref{Eq::Condition}). Finally, the metric is $\omega_1^2+\omega_2^2=\mu^2dx^2+\nu^2dy^2$.
\end{proof}

\subsection{Resolution of the system (\ref{Eq::Condition})}
We now focus on the resolution of \eqref{Eq::Condition} and assume that $\psi_1\neq 0\ [\pi/2]$ or $\psi_2\neq 0$ so that $c_1$ or $c_2\neq 0$ (if $\psi_1=0\ [\pi/2]$ and $\psi_2=0$ then $\psi=0\ [\pi]$ or $\psi=\pi/2\ [\pi]:$ the first case is studied in Appendix \ref{app case psi 0 mod pi} and the second case in Theorem \ref{thm psi pi sur 2} in Section \ref{section product} below). Let us first observe that the resolution of this system is then equivalent to the resolution of the single hyperbolic PDE
\begin{equation} \label{PDE-A}
 \partial^2_{x y}\zeta = -c_1 c_2\ \zeta
\end{equation}
for $\zeta=\mu$ or $\nu,$ which is a 1-dimensional Klein-Gordon equation. Indeed, if $\mu$ and $\nu$ satisfy (\ref{Eq::Condition}) then they obviously also satisfy (\ref{PDE-A}). Conversely, assuming first that $c_1\neq 0,$ if $\mu$ is a solution of (\ref{PDE-A}) we obtain a solution $\mu,\nu$ of (\ref{Eq::Condition}) by setting $\nu:=-\frac{1}{c_1}\partial_y\mu,$ and, similarly, if $c_2\neq 0$ and $\nu$ is a solution of (\ref{PDE-A}) we obtain a solution $\mu,\nu$ of (\ref{Eq::Condition}) by setting $\mu:=\frac{1}{c_2}\partial_x\nu.$ We finally note that we may solve a Cauchy problem for (\ref{PDE-A}): let us fix a smooth regular curve $\Gamma=(\gamma_1,\gamma_2)$ in the coordinate plane which does not intersect any line parallel to the coordinate axes in more than one point and let us consider two smooth functions $f,g$ on $\Gamma$; then there exists a unique solution $\zeta$ of (\ref{PDE-A}) such that 
\begin{equation} \label{Eq::InitialCond}
\left\lbrace
\begin{array}{rl}
 \left. \zeta \right|_\Gamma &= f \\[8pt]
 \left. \partial_n \zeta \right|_\Gamma &=g,
\end{array}
\right.
\end{equation}
where $\partial_n$ denotes differentiation with respect to the direction normal to the curve. If the above geometric condition on $\Gamma$ is not satisfied, the Cauchy problem is in general insoluble. It thus appears that a general constant angle surface in $\R^{1,3}$ locally depends on two arbitrary real functions of one real variable (the initial conditions $f,g$ of the Cauchy problem (\ref{PDE-A})-(\ref{Eq::InitialCond}) for $\mu$ or for $\nu$). Details on this Cauchy problem and its explicit resolution using a Bessel function may be found in \cite[Chapter II]{Kosh}.

\section{A frame adapted to a constant angle surface}\label{section adapted frame}

With the last representation theorem and the explicit expression of the lift $g$ of the Gauss map (Lemma \ref{lift G}), we can  construct a special orthonormal frame adapted to a given constant angle surface. We will first give a geometric construction of this frame, and then its explicit expression. We finally use these results to obtain easily the second order invariants of the surface.

\subsection{Geometric construction of an adapted frame}
Let us assume that the immersion is given as in Theorem \ref{thm representation 2}. We consider the orthonormal frame tangent to the immersion
\begin{equation}\label{dF Ti}
T_1:=\frac{1}{\mu}\partial_xF=\frac{1}{\mu}\xi(\partial_x),\hspace{1cm} T_2:=\frac{1}{\nu}\partial_yF=\frac{1}{\nu}\xi(\partial_y).
\end{equation}
Let us note that $(T_1,T_2)$ is positively oriented: by the very definitions of $\xi,$ $\omega_1$ and $\omega_2$ in Theorem \ref{thm representation 2} we have
\begin{eqnarray*}
\partial_xF&=&\xi(\partial_x)\\
&=&g^{-1}(\omega_1(\partial_x)J+\omega_2(\partial_x)K)\widehat{g}\\
&=&\mu g^{-1}(\sin u\ J-\cos u\ K)\widehat{g}
\end{eqnarray*} 
and similarly
$$\partial_yF=\nu g^{-1}(\cos u\ J+\sin u\ K)\widehat{g}.$$
We thus have, in $\HC,$
$$\partial_x F\ \widehat{\partial_yF}=\mu\nu g^{-1}Ig=\mu\nu G,$$
which implies that 
$$T_1\cdot T_2=G,$$
and thus that $(T_1,T_2)$ is positively oriented.

Let us recall that the curvature ellipse at a point $x_o\in M$ is the ellipse in the normal plane
$$\{\mathbf{II}(w,w):\ w\in T_{x_o}M,\ |w|=1\}\subset N_{x_o}M.$$
\begin{prop}
The curvature ellipse is a segment $[\frac{2}{\mu} N_1,\frac{2}{\nu} N_2]$ where $N_1,N_2$ are normal vectors such that $|N_1|^2=-|N_2|^2=1$ and $\langle N_1,N_2\rangle=0.$ Moreover
\begin{equation}\label{ff Ti}
\mathbf{II}(T_1,T_1)=\frac{2}{\mu} N_1,\hspace{.5cm}\mathbf{II}(T_2,T_2)=\frac{2}{\nu} N_2\hspace{.5cm}\mbox{and}\hspace{.5cm}\mathbf{II}(T_1,T_2)=0.
\end{equation}
The vector $N_2$ is future-directed and $(N_2,N_1)$ is a positively oriented basis of the plane normal to $M.$
\end{prop}
We thus obtain a natural moving frame $(N_2,N_1,T_1,T_2)$ adapted to the constant angle surface. This frame is moreover positively oriented in $\R^{1,3}$ and such that its first vector is future-directed.
\begin{proof}
Since $K_N=0$ the curvature ellipse is a segment $[\alpha,\beta]\subset N_{x_o}M$ and since $K=0,$ $\langle\alpha,\beta\rangle=0$ (by the Gauss equation). Let us also note that $\alpha,\beta\neq 0$ since  $\mathbf{II}$ is not degenerate (if for instance $\alpha=0$ and $w\in T_{x_o}M,$ $|w|=1$ is such that $\mathbf{II}(w,w)=\alpha,$ we would have $\mathbf{II}(w,w)=\mathbf{II}(w,w^{\perp})=0$ (the curvature ellipse is a segment with extremal point $\mathbf{II}(w,w)$) and thus $dG_{x_o}(w)=0,$ in contradiction with $G'(x_o)\neq 0$). Let us first show that $\mathbf{II}(T_1,T_2)=0.$ Differentiating $G=T_1\wedge T_2,$ we easily get, for all $w\in T_{x_o}M,$
\begin{equation}\label{dG w}
dG(w)=\mathbf{II}\left(T_1,w\right)\wedge T_2+T_1\wedge \mathbf{II}\left(T_2,w\right).
\end{equation}
But we also have
\begin{equation}\label{dG wp}
dG(w)=G'w=2J\left\{-\sin\left(\frac{2z}{\sin\psi}\right)+\cos\left(\frac{2z}{\sin\psi}\right)I\right\}w.
\end{equation}
For $w=T_1\simeq \frac{1}{\mu}\in\C$ we get $H(dG(T_1),dG(T_1))=\frac{4}{\mu^2}$ which is equivalent to
\begin{equation}\label{DGT1 scal wedge}
\langle dG(T_1),dG(T_1)\rangle=\frac{4}{\mu^2}\hspace{.5cm}\mbox{and}\hspace{.5cm}dG(T_1)\wedge dG(T_1)=0.
\end{equation}
In view of (\ref{dG w}) with $w=T_1,$ the second property reads
\begin{equation}\label{wedge II 0}
\mathbf{II}\left(T_1,T_1\right)\wedge \mathbf{II}\left(T_1,T_2\right)=0.
\end{equation}
Since $\mathbf{II}\left(T_1,T_1\right)$ belongs to the curvature ellipse $[\alpha,\beta]$ and $\mathbf{II}\left(T_1,T_2\right)$ is tangent to the ellipse,
we can write
$$\mathbf{II}\left(T_1,T_1\right)=\alpha+\lambda(\beta-\alpha)\hspace{.5cm}\mathbf{II}\left(T_1,T_2\right)=\lambda'(\beta-\alpha)$$
for some $\lambda,\lambda'\in \R,$ and (\ref{wedge II 0}) then implies $\lambda'\alpha\wedge\beta=0.$ This in turn implies $\lambda'=0:$ by contradiction, if $\lambda'\neq 0$ we would obtain $\alpha\wedge\beta=0$ and since $\langle\alpha,\beta\rangle=0,$ $\alpha$ and $\beta$ would be collinear null vectors; the norm of $dG(T_1)=\alpha\wedge T_2+T_1\wedge\lambda'(\beta-\alpha)$ would then be zero, in contradiction with (\ref{DGT1 scal wedge}). Thus $\lambda'=0$ and $\mathbf{II}(T_1,T_2)=0.$

Since $\mathbf{II}(T_1,T_2)=0$ we get that $\mathbf{II}(T_1,T_1)$ and $\mathbf{II}(T_2,T_2)$ are the extremal points of the curvature ellipse, and we assume that
$$\alpha=\mathbf{II}(T_1,T_1)\hspace{.5cm}\mbox{and}\hspace{.5cm}\beta=\mathbf{II}(T_2,T_2).$$
We deduce from (\ref{dG w}) that $dG(T_1)=\alpha\wedge T_2$ and from (\ref{DGT1 scal wedge}) that $|\alpha|^2=\frac{4}{\mu^2}.$ So there exists a unit spacelike vector $N_1$ such that $\alpha=\frac{2}{\mu}N_1.$ Similarly, we obtain from (\ref{dG w}) and (\ref{dG wp}) with $w=T_2\simeq\frac{i}{\nu}\in\C$ that
$$dG(T_2)=T_1\wedge \beta=\frac{i}{\nu}G',$$
which implies that $|\beta|^2=-\frac{4}{\nu^2},$ and thus that there exists a unit timelike vector $N_2$ such that $\beta=\frac{2}{\nu}N_2.$ Let us finally show that $(N_2,N_1)$ is a positively oriented basis of $N_{x_o}M$ with $N_2$ future-oriented: by (\ref{ff Ti}), (\ref{dG w}) and (\ref{dG wp}) with $w:=T_1\simeq\frac{1}{\mu}$ and $w:=T_2\simeq\frac{i}{\nu}$ we obtain
$$2N_1\cdot T_2=G'\hspace{.5cm}\mbox{and}\hspace{.5cm}2T_1\cdot N_2=iG'. $$
Now, we have
$$G'=2J\left(-\sin\left(\frac{2z}{\sin\psi}\right)+\cos\left(\frac{2z}{\sin\psi}\right) I\right)$$
and thus $G'^2=-4$ in $\HC;$ this implies that, in $\HC,$
$$N_2\cdot N_1\cdot T_1\cdot T_2=i,$$
which is also the canonical volume form $e_0^o\cdot e_1^o\cdot e_2^o\cdot e_3^o$ of $\R^{1,3}.$ The basis $(N_2,N_1,T_1,T_2)$ is thus positively oriented in $\R^{1,3},$ and so is $(N_2,N_1)$ in $N_{x_o}M.$ The vector $N_2$ is future-directed: we have
\begin{eqnarray*}
\vec H&=&\frac{1}{2}\left(\mathbf{II}(T_1,T_1)+\mathbf{II}(T_2,T_2)\right)\\
&=&\frac{1}{\mu}N_1+\frac{1}{\nu}N_2\\
&=&(h_0\sinh v+h_1\cosh v)N_1+(h_0\cosh v+h_1\sinh v)N_2
\end{eqnarray*}
by (\ref{mu nu h0 h1}), i.e.
$$\vec H=h_0(\sinh v\ N_1+\cosh v\ N_2)+h_1(\cosh v\ N_1+\sinh v\ N_2).$$
 Since $h_0$ and $h_1$ are by hypothesis the coordinates of $\vec H$ in a normal basis which is positively oriented in space and in time, the vector $\sinh v\ N_1+\cosh v\ N_2$ is future-directed, and so is $N_2.$ This proves the proposition.
\end{proof}

\subsection{Explicit expression of the adapted frame.}\label{section explicit expression frame}

We only give here results of calculations, and refer to Appendix \ref{App::CompFrame} for more details. Direct computations using the special lift (\ref{explicit g}) of the Gauss map and the representation formula (\ref{omegas u mu nu})-(\ref{xi g omegas}) give the following explicit formulas:
$$T_1=
\left(
-\sinh \left(\psi_2\right) 
\cosh \left( \varphi_2 \right),
- \sinh \left(\psi_2 \right)  \sinh \left( \varphi_2 \right),
 \cosh \left( \psi_2 \right)  \sin \left(\varphi_1 \right),
\cosh\left(\psi_2 \right)\cos \left(\varphi_1 \right)
\right)$$
and 
$$T_2=\left(
 \sin \left(\psi_1\right) 
\sinh \left( \varphi_2 \right),
\sin \left(\psi_1 \right)  \cosh \left( \varphi_2 \right),
-  \cos \left( \psi_1 \right)  \cos \left(\varphi_1 \right),
 \cos \left(\psi_1 \right) \sin \left(\varphi_1 \right)
\right)$$
where
$$\psi=\psi_1+i\psi_2\hspace{.5cm}  \mbox{and}\hspace{.5cm}\varphi:=\dfrac{2 z}{\sin \psi} =\varphi_1+i\varphi_2.$$
Let us note that (\ref{nabla Ti}) and (\ref{ff Ti}) imply that
\begin{equation}\label{formulas dT}
dT_1=(c_1 T_2 + 2 N_1)dx+c_2 T_2dy,\hspace{.5cm}dT_2=-c_1T_1dx+(-c_2 T_1 + 2 N_2)dy
\end{equation}
which may naturally also be obtained by direct computations. Similarly, we also have 
$$N_1=\left( \cos (\psi_1) \sinh (\varphi_2 ),\cos (\psi_1) \cosh (\varphi_2 ),\sin (\psi_1) \cos (\varphi_1),-\sin (\psi_1) \sin (\varphi_1) \right),$$
$$N_2= \left(\cosh (\psi_2) \cosh (\varphi_2),\cosh (\psi_2) \sinh (\varphi_2),-\sinh (\psi_2)\sin (\varphi_1),-\sinh (\psi_2)\cos (\varphi_1)\right),$$
\begin{equation}\label{formulas dN}
dN_1  = (-2 T_1 +c_2 N_2)dx-c_1 N_2dy,\hspace{.5cm} dN_2 =c_2 N_1 dx+(2 T_2-c_1 N_1)dy,
\end{equation}
and thus
\begin{equation}\label{nablap Ni}
\nabla' N_1=(c_2dx-c_1dy)N_2,\hspace{1cm}\nabla' N_2=(c_2dx-c_1dy)N_1.
\end{equation}
It appears on these formulas that the special frame $(T_1,T_2,N_1,N_2)$ only depends on the constant angle $\psi$ and the value of the parameter $z.$ We have by (\ref{dF Ti})
\begin{equation}\label{F mu nu Ti}
F=\int \mu T_1 dx+\nu T_2 dy,
\end{equation}
and the immersion with constant angle $\psi$ with respect to $e_2^o\wedge e_3^o$ is entirely determined by $\mu$ and $\nu,$ in accordance with Theorem \ref{thm representation 2}. 

\subsection{Second order invariants of a constant angle surface.}\label{section invariants}
Let us fix a point $x_o\in M$ and consider the quadratic form $\delta:T_{x_o}M\rightarrow\R$ defined by
$$\delta:=\frac{1}{2}dG_{x_o}\wedge dG_{x_o}$$
where $\Lambda^4\R^{1,3}$ is naturally identified with $\R$ as in Section \ref{section model grassm}. Let us recall from \cite{BaySB} that $\delta$ determines the asymptotic directions of the surface at $x_o,$ and also the fourth numerical invariant of the surface at that point 
$$\Delta:=disc\ \delta$$
(the other three invariants are $K,$ $K_N$ and $|\vec{H}|^2$). We recall that a non-zero vector $w\in T_{x_o}M$ is said to be an asymptotic direction of $M$ at $x_o$ if $\delta(w)=0;$ the existence of a pair of asymptotic directions is thus determined by the sign of $\Delta.$ In the next lemma we compute $\delta$ and the invariant $\Delta$ for a constant angle surface:
\begin{lem}
The matrix of $\delta$ in the orthonormal basis $(T_1,T_2)$ of $T_{x_o}M$ is
$$Mat(\delta,(T_1,T_2))=\frac{2}{\mu\nu}\left(\begin{array}{cc}
0&1\\
1&0
\end{array}\right).$$
In particular $\Delta=-\frac{4}{\mu^2\nu^2}$ and $T_1,T_2$ are the asymptotic directions of $M$ at $x_o.$ 
\end{lem}
\begin{proof}
By (\ref{ff Ti}) and (\ref{dG w}), we have
$$dG(T_1)=\mathbf{II}(T_1,T_1)\wedge T_2=\frac{2}{\mu}\ N_1\wedge T_2$$
and
$$dG(T_2)=T_1\wedge \mathbf{II}(T_2,T_2)=\frac{2}{\nu}\ T_1\wedge N_2.$$
Thus $\delta(T_1,T_1)=\delta(T_2,T_2)=0$ and
$$\delta(T_1,T_2)=\frac{2}{\mu\nu}\ N_2\wedge N_1\wedge T_1\wedge T_2\simeq \frac{2}{\mu\nu}$$
since $(N_2,N_1,T_1,T_2)$ is a positively oriented and orthonormal basis of $\R^{1,3}$ with $N_2$ timelike and future-oriented.
\end{proof}
Let us also mention that the mean curvature vector of a constant angle surface is given by
$$\vec{H}=\frac{1}{2}\left(\mathbf{II}(T_1,T_1)+\mathbf{II}(T_2,T_2)\right)=\frac{1}{\mu}N_1+\frac{1}{\nu}N_2$$
and thus that 
\begin{equation}\label{H mu nu}
|\vec{H}|^2=\frac{1}{\mu^2}-\frac{1}{\nu^2}.
\end{equation}
We thus have the following
\begin{prop}
The four invariants of $M$ are
$$K=K_N=0,\hspace{.5cm} |\vec{H}|^2=\frac{1}{\mu^2}-\frac{1}{\nu^2}\hspace{.5cm}\mbox{and}\hspace{.5cm} \Delta=-\frac{4}{\mu^2\nu^2}.$$
\end{prop}

\section{Incompleteness of the constant angle surfaces.}\label{section complete}

\begin{prop}
An oriented spacelike surface in $\R^{1,3}$ with regular Gauss map and constant angle $\psi\neq 0\ [\pi/2]$ is not complete.
\end{prop}
\begin{proof}
Recalling (\ref{c1 c2 psi1 psi2}), the property $\psi\neq 0\ [\pi/2]$ is equivalent to $c_1$ or $c_2\neq 0.$ Let us first assume that $c_1\neq 0.$ Since
$$\langle\frac{1}{\mu}T_1,\partial x\rangle=1\hspace{.5cm}\mbox{and}\hspace{.5cm}\langle\frac{1}{\mu}T_1,\partial y\rangle=0,$$
the gradient of the function $x$ on $\mathcal{U}$ is $\nabla x=\frac{1}{\mu}T_1.$ Let us consider its norm 
$$f(t):=|\nabla x|=|\mu|^{-1}$$ 
along an integral curve of $T_2.$ By the first equation in (\ref{Eq::Condition}) it satisfies
$$(f^2)'=-2\mu^{-3}d\mu(T_2)=2c_1\mu^{-3}=2c_1f^3.$$
Since $f$ does not vanish, this equation implies that the flow of $T_2$ cannot be defined for all $t\in\R,$ and thus that the surface is not complete. If $c_1=0$ and $c_2\neq 0,$  we analogously consider
$$g(t)=|\nabla y|=|\nu|^{-1}$$ 
along an integral curve of $T_1.$ It satisfies $(g^2)'=2c_2g^3,$ which also implies that the surface is not complete.
\end{proof}
\begin{rem}The surface $H^1(r_1)\times S^1(r_2)\subset\R^{1,1}\times\R^2=\R^{1,3}$ with
$$H^1(r_1):=\{(x_0,x_1)\in\R^{1,1}:\ x_0^2-x_1^2=r_1^2,\ x_0>0\},$$ 
more generally a product of two regular and complete curves $\gamma_1\times\gamma_2\subset\R^{1,1}\times\R^2$ ($\gamma_1$ spacelike), is a complete spacelike surface with regular Gauss map and constant angle $\psi=\pi/2\ [\pi].$ In fact, in view of the proposition and of Theorem \ref{thm psi pi sur 2} below, up to a congruence, all the complete surfaces in $\R^{1,3}$ with constant complex angle and regular Gauss map are of that form.
\end{rem}

\section{Characterization of constant angle surfaces which are product of plane curves}\label{section product}

We show here that the surfaces of constant angle $\pi/2\ [\pi]$ are product of plane curves. More precisely, we have the following result:

\begin{thm}\label{thm psi pi sur 2}
A surface has constant angle $\psi=\pi/2\ [\pi]$ with respect to a spacelike plane $p_o$ if and only if it is a product $\gamma_1\times\gamma_2$ of curves in the perpendicular planes $p_o$ and $p_o^{\perp}.$ 
\end{thm}
\begin{proof}
We assume that $p_o=e_3^o\wedge e_4^o.$ Since $\psi_1=\pi/2\ [\pi]$ and $\psi_2=0,$ we have $c_1=c_2=0$ (see (\ref{c1 c2 psi1 psi2})) and (\ref{Eq::Condition}) implies that $\mu$ only depends on $x,$ and $\nu$ only depends on $y;$ moreover, the explicit formulas for $T_1$ and $T_2$ in Section \ref{section explicit expression frame} read
$$T_1=(0,0,\sin 2x,\cos 2x)\hspace{.5cm}\mbox{and}\hspace{.5cm} T_2=(-\sinh 2y,\cosh 2y,0,0).$$  
Formula (\ref{F mu nu Ti}) gives the result. Reciprocally, for a product $\gamma_1\times \gamma_2$ in $\R^2\times\R^{1,1}$ where $\gamma_2$ is a spacelike curve, setting $p_o=\R^2\times\{0\},$ $T_1=\gamma_1'/|\gamma_1'|$ and $T_2=\gamma_2'/|\gamma_2'|$ the angle $\psi$ between the surface $\gamma_1\times \gamma_2$ and $p_o$ is by definition such that
$$\cos\psi=H(p_o,T_1\wedge T_2)=\langle p_o, T_1\wedge T_2\rangle+i\ p_o\wedge T_1\wedge T_2=0;$$
this implies that $\psi=\pi/2\ [\pi].$
\end{proof}
Since surfaces with constant angle $\psi=0\ [\pi]$ were described in Remark \ref{rem psi degenerate hyperplane}, this result completes the description of the surfaces with constant angle $\psi= 0\ [\pi/2].$
\section{Characterization of constant angle surfaces with real or pure imaginary constant angle}\label{section real im angle}
We describe in the following theorem the spacelike surfaces with regular Gauss map and real or pure imaginary constant angle in $\R^{1,3}.$ Let us recall that a \emph{holonomy tube} over a spacelike curve $\gamma\in\R^{1,3}$ is a surface obtained by the normal parallel transport along $\gamma$ of some curve $c$ initially given in a fixed hyperplane normal to $\gamma:$ $c$ is the curve of the \emph{starting points} of the tube.
\begin{thm}
A surface with real constant angle $\psi=\psi_1\in\R$ (resp. pure imaginary constant angle $\psi=i\psi_2\in i\R$) with respect to a spacelike plane $p_o$ is a holonomy tube over a plane curve $\gamma\in p_o.$ Moreover, if $c\in N_{m_o}\gamma$ is the curve of the starting points of the tube, then $c$ is an helix curve in $N_{m_o}\gamma\simeq\R^{1,2}$ with respect to a spacelike (resp. timelike) direction.
\end{thm}
\begin{proof}
Let us fix $m_o=(x_o,y_o)\in\mathcal{U}\subset\R^2$ and consider the curves $\gamma(x):=F(x,y_o)$ and $c(y):=F(x_o,y).$ We assume that $p_o$ is the plane $\{0\}\times\R^2\subset\R^{1,1}\times\R^2=\R^{1,3}.$ The curve $\gamma$ is everywhere tangent to 
$$T_1=(0,0,\sin\left(\frac{2x}{\sin\psi_1}\right),\cos\left(\frac{2x}{\sin\psi_1}\right)),$$
by the expression of $T_1$ with $\psi_2=0$ in Section \ref{section explicit expression frame}, and is thus a curve in $p_o.$ The curve $c$ belongs to the hyperplane normal to the curve $\gamma$ at $x_o.$ Indeed, 
\begin{eqnarray*}
\langle c(y)-\gamma(x_o),\gamma'(x_o)\rangle&=&\langle F(x_o,y)-F(x_o,y_o),\partial_xF(x_o,y_o)\rangle\\
&=&\langle\int_{y_o}^y\partial_yF(x_o,t)dt,\partial_xF(x_o,y_o)\rangle\\
&=&\int_{y_o}^y\mu(x_o,y_o)\nu(x_o,t)\langle T_2(x_o,t),T_1(x_o,y_o)\rangle,
\end{eqnarray*}
which is zero since $T_1(x_o,y_o)=T_1(x_o,t)$ is orthogonal to $T_2(x_o,t)$ ($T_1(x_o,t)$ does not depend on $t,$ by the expression of $T_1$ above). Finally, if we fix $y=y_1,$ the curve $x\mapsto F(x,y_1)$ may be regarded as a normal section of $\gamma;$ it is parallel since
$$\partial_xF(x,y_1)=\mu(x,y_1)T_1(x,y_1)=\mu(x,y_1)T_1(x,y_o)$$
is tangent to $\gamma$ at $x.$ The curve $c$ is an helix in $N_{m_o}\gamma:$ its unit tangent is 
$$y\mapsto T_2(x_o,y)=(\sin\psi_1\sinh\varphi_2,\sin\psi_1\cosh\varphi_2,-\cos\psi_1\cos\varphi_1,\cos\psi_1\sin\varphi_1),$$ 
and if $\vec{A}$ is the fixed direction $(0,0,-\cos\varphi_1,\sin\varphi_1)$ in  $N_{m_o}\gamma$ we have
$$\langle T_2(x_o,y),\vec{A}\rangle=\cos\psi_1.$$
The curve $c$ is thus a constant angle curve in $\R^{1,2}$ with respect to the spacelike direction $\vec{A};$ the constant  angle is $\psi_1.$ The proof for $\psi=i\psi_2\in\ i\R$ is analogous. 
\end{proof}

\section{Characterization of constant angle surfaces in hyperspheres and lightcones}\label{section spheres lightcones}
\subsection{The immersion in the adapted orthonormal frame }
Let us write the immersion of a constant angle surface in the form 
\begin{equation}\label{F fi}
\bfz=f\ T_1+\tilde{f}\ T_2+g\ N_1+\tilde{g}\ N_2
\end{equation}
where $(T_1,T_2,N_1,N_2)$ is the frame adapted to the surface introduced in Section \ref{section adapted frame} and $f,\tilde{f},g$ and $\tilde{g}$ are smooth real functions of the variables $x$ and $y.$ Formulas (\ref{nabla Ti}) and (\ref{ff Ti}) imply the following:
\begin{thm}\label{Th::CoordintateForm}
We assume that $\psi_1\neq 0\ [\pi/2]$ and $\psi_2\neq 0$ (i.e. $c_1,c_2\neq 0$) and suppose that $\mu$ and $\nu$ are solutions of \eqref{Eq::Condition}. Then the immersion reads
\begin{equation}\label{F f ft g gt}
\bfz=f\ T_1 +\frac{\partial_y f}{c_2}\ T_2+g\ N_1+\frac{\partial_y g}{c_1}\ N_2
\end{equation}
where $f$ and $g$ are solutions of
\begin{equation}\label{EQ::Systemfi}
\left\lbrace
\begin{array}{rlrl}
\partial_x f&= \mu + \dfrac{(4+c_1^2) g -\partial^2_{yy} g}{2},  &
\partial_y f&= \dfrac{c_2 (c_1^2 g- \partial^2_{yy}g)}{2c_1},  \\[12pt]
\partial_x g&= -\dfrac{c_2}{2}\nu + \dfrac{(c_2^2-4)f + \partial^2_{yy} f}{2}, &
\partial_y g&=\dfrac{c_1}{2}\nu-\dfrac{c_1(c_2^2 f+\partial^2_{yy}f)}{2 c_2}, \\[12pt]
\partial^2_{xy} f &=-c_1 c_2 f,   &\partial^2_{xy} g &=-c_1c_2 g.
\end{array}
\right.
\end{equation}
Reciprocally, given two functions $f,g$ solving this PDE system for $\mu$ and $\nu$ solutions of \eqref{Eq::Condition}, the immersion $\bfz$ given by \eqref{F f ft g gt} is a spacelike immersion in $\R^{1,3}$ of constant complex angle.
\end{thm}
The advantage of this formulation lies in the fact that the solutions of this system directly give the immersion; in the previous formulations, Theorems \ref{thm representation} and \ref{thm representation 2}, a last integration was still required to obtain the immersion $F$ from the 1-form $\xi.$

\begin{proof}
If $F$ is an immersion of constant angle $\psi$ and metric $\mu^2 dx^2+\nu^2dy^2,$ we define the functions 
\[
f=\langle F, T_1 \rangle, \quad
\tilde{f}=\langle F, T_2 \rangle, \quad
g=\langle F, N_1 \rangle \quad \text{and} \quad
\tilde{g}=-\langle F, N_2 \rangle; \quad
\]
they are such that (\ref{F fi}) holds. Using  (\ref{nabla Ti}) and (\ref{ff Ti}) we compute
\begin{eqnarray*}
df&=&\langle dF,T_1\rangle+\langle F,dT_1\rangle\\
&=& \langle \mu T_1 dx+\nu T_2 dy,T_1\rangle+\langle F,(c_1dx+c_2dy)T_2+2N_1dx\rangle\\
&=&(\mu+c_1\tilde{f}+2g)dx+c_2\tilde{f}dy
\end{eqnarray*}
and similarly
$$d\tilde{f}=-c_1 fdx+(\nu-c_2 f-2\tilde{g})dy,\ dg=(-2 f-c_2 \tilde{g})dx+c_1\tilde{g} dy$$
and
$$d\tilde{g}= -c_2 g dx+(c_1g-2 \tilde{f})dy.$$
This implies that $\tilde{f}=\frac{\partial_y f}{c_2}$, $\tilde{g}=\frac{\partial_y g}{c_1}$ and $f,g$ satisfy \eqref{EQ::Systemfi}.

Reciprocally, if $f,g$ are solutions of (\ref{EQ::Systemfi}), straightforward computations using (\ref{formulas dT}) and (\ref{formulas dN}) show that the function $F$ defined by (\ref{F f ft g gt}) satisfies $\partial_xF=\mu T_1$ and $\partial_yF=\nu T_2;$ it is thus an immersion with Gauss map $G=T_1\wedge T_2,$ and, since $H(I,G)=\cos\psi$ (by the explicit formulas for $T_1$ and $T_2$ in Section \ref{section explicit expression frame}), it is of constant angle $\psi$ with respect to $p_0:=I.$
\end{proof}
\begin{rem}
The system \eqref{EQ::Systemfi} is in fact equivalent to the smaller system formed by the first four equations and one of the last two equations: the sixth equation may indeed be easily obtained from these five equations.
\end{rem}

\subsection{Description of the constant angle surfaces in hyperspheres and lightcones}
We determine here the constant angle surfaces in hyperspheres and lightcones, i.e., up to translations, immersions of constant angle and constant norm.
\begin{cor}\label{cor sphere lightcone}
Keeping the notation introduced above, we have the following:
\\1. The constant angle immersion $F$ belongs to a hypersphere if and only if 
$$\mu=  r_1 e^{c_1 y-c_2 x}+ r_2 e^{-(c_1 y-c_2 x)}$$
and
$$\nu=-r_1 e^{c_1 y-c_2 x}+ r_2 e^{-(c_1 y-c_2 x)}$$
for some constants $r_1,r_2 \neq 0$. In that case, and up to a translation, we have $|\bfz|^2= r_1 r_2.$
\\
\\2.  The constant angle immersion $F$ belongs to a lightcone if and only if 
\begin{equation}\label{mu lightcone}
\mu= r e^{\epsilon(c_1 y-c_2x)}\hspace{.5cm}\mbox{and}\hspace{.5cm}\nu= -\epsilon r e^{\epsilon(c_1 y-c_2x)}
\end{equation}
for some constant $r \neq 0$ and $\epsilon=\pm 1$.
\\
\\Moreover, if the immersion belongs to a hypersphere or a lightcone it can be written, up to a translation, in the form
\[
\bfz =\dfrac{-\mu N_1 + \nu N_2}{2} 
\]
where $N_1,$ $N_2\in\R^{1,3}$ are the unit orthogonal vector fields introduced in Section \ref{section adapted frame}.
\end{cor}
\begin{proof} 
From Theorem \ref{Th::CoordintateForm}, we can write the immersion $F$ in the form 
\[
\bfz=f\ T_1 +\frac{\partial_y f}{c_2}\ T_2+g\ N_1+\frac{\partial_y g}{c_1}\ N_2
\]
with $f$ and $g$ solutions of \eqref{EQ::Systemfi}. But $F$ has constant norm if and only if $\langle \bfz , \bfz_x \rangle=\langle \bfz , \bfz_y\rangle=0$, that is $f=0$. The function $g$ is thus a solution of
\[
\left\lbrace
\begin{array}{c}
0= \mu + \dfrac{(4+c_1^2) g -\partial^2_{yy} g}{2},  
\qquad 0= c_1^2 g- \partial^2_{yy}g,  \\[12pt]
\partial_x g= -\dfrac{c_2}{2}\nu, \qquad
\partial_y g=\dfrac{c_1}{2}\nu,\qquad  \partial^2_{xy} g =-c_1c_2 g. 
\end{array}
\right.
\]
Using \eqref{Eq::Condition} we have
\[
\partial_y g = \dfrac{c_1}{2}\nu= -\dfrac{\partial_y\mu}{2}
\]
and obtain that
$$g(x, y) =-\dfrac{\mu(x, y)}{2}+ t(y)$$ 
for some function $t;$ but the last condition now reads
\[
\partial^2_{xy}g =-\dfrac{1}{2}\partial^2_{xy}\mu=\dfrac{1}{2}c_1 c_2 \mu -c_1 c_2 t,
\]
and since $\partial^2_{xy}\mu=-c_1 c_2 \mu$ we get $t=0$. With $g=-\dfrac{\mu}{2}$, the system \eqref{EQ::Systemfi}, without duplicities, reads
\[
\begin{array}{c}  
\qquad \partial^2_{yy}\mu=c_1^2 \mu,  \qquad
\partial_x\mu= c_2 \nu, \qquad
\partial_y\mu =-c_1\nu, \qquad
 \partial^2_{xy} \mu =-c_1c_2 \mu
\end{array}
\]
whose solutions are 
$$\mu= r_1 e^{c_1 y-c_2 x}+  r_2 e^{-(c_1 y-c_2 x)}\qquad \text{and} \qquad \nu= -r_1 e^{c_1 y-c_2 x}+ r_2 e^{-(c_1 y-c_2 x)}$$
where $r_1$ and $r_2$ are real numbers. We thus obtain
\[
\bfz=gN_1+\frac{\partial_yg}{c_1}N_2=-\dfrac{\mu}{2}N_1-\dfrac{\partial_y\mu}{2c_1}N_2=\dfrac{-\mu N_1 + \nu N_2}{2} 
\]
whose norm is
\[
|\bfz|^2=\dfrac{\mu^2-\nu^2}{4}=r_1 r_2.
\]
Finally, the immersion belongs to the lightcone at the origin if and only if $r_1$ or $r_2=0,$ which implies the last claim in the statement.
\end{proof}
\begin{rem}
The second part of the corollary in particular shows that the surfaces in (\ref{F new example}) are, up to a congruence and scaling, the unique surfaces with constant angle in a lightcone.
\end{rem}

\begin{cor} Let us assume that $M$ is a spacelike surface of constant angle $\psi\neq 0\ [\pi/2].$ Then the following properties are equivalent:
\begin{enumerate}
\renewcommand{\theenumi}{\alph{enumi}} \renewcommand{\labelenumi}{\theenumi)}
\item \label{H par pty 1} $\vec{H}\in\Gamma(NM)$ is parallel;
\item \label{H par pty 2} $\vec{H}$ is everywhere lightlike;
\item \label{H par pty 3} the parameter $z=x+iy\in\mathcal{U}$ is conformal, i.e. $\mu=\nu$ on $\mathcal{U}.$
\end{enumerate}
If one of these properties occurs then the surface belongs to a lightcone and is, up to a congruence, of the form (\ref{F new example}). 
\end{cor}
\begin{proof}
Since $\beta=u+iv$ is an holomorphic function and by (\ref{mu nu ux uy}) we have
$$\partial_yv=\partial_xu=-\frac{1}{\nu}\partial_y\mu.$$
Recalling (\ref{mu nu h0 h1}) we deduce that
\begin{eqnarray*}
-\frac{\partial_y\mu}{\mu^2}&=&\partial_yh_0\sinh v+\partial_yh_1\cosh v+(h_0\cosh v+h_1\sinh v)\partial_yv\\
&=&\partial_yh_0\sinh v+\partial_yh_1\cosh v-\frac{1}{\nu^2}\partial_y\mu
\end{eqnarray*}
and using (\ref{H mu nu}) that
\begin{equation}\label{pmu ph}
-|\vec{H}|^2\partial_y\mu=\partial_yh_0\sinh v+\partial_yh_1\cosh v.
\end{equation}
We similarly have
\begin{eqnarray*}
-\frac{\partial_x\nu}{\nu^2}&=&\partial_xh_0\sinh v+\partial_xh_1\cosh v+(h_0\cosh v+h_1\sinh v)\partial_xv\\
&=&\partial_xh_0\sinh v+\partial_xh_1\cosh v-\frac{1}{\mu^2}\partial_x\nu
\end{eqnarray*}
and thus
\begin{equation}\label{pnu ph}
|\vec{H}|^2\partial_x\nu=\partial_xh_0\cosh v+\partial_xh_1\sinh v.
\end{equation}
Since, by (\ref{mu nu ux uy}) and (\ref{ux uy psi}), we have
$$\frac{1}{\nu}\partial_y\mu+i\frac{1}{\mu}\partial_x\nu=-\partial_xu+i\partial_yu=2\cot\psi,$$
(\ref{pmu ph}) and (\ref{pnu ph}) yield
$$ 2|\vec{H}|^2\cot\psi =-\frac{1}{\nu}(\partial_yh_0\sinh v+\partial_yh_1\cosh v)+i\frac{1}{\mu}(\partial_xh_0\cosh v+\partial_xh_1\sinh v).$$
Since $h_0$ and $h_1$ are the components of the mean curvature vector $\vec{H}$ in a parallel frame normal to the surface, this formula implies the equivalence between (\ref{H par pty 1}) and (\ref{H par pty 2}). The equivalence between (\ref{H par pty 2}) and (\ref{H par pty 3}) is a consequence of (\ref{H mu nu}). If now the properties (\ref{H par pty 1})-(\ref{H par pty 3}) hold then $\mu=\nu$ and the system (\ref{Eq::Condition}) easily implies that $\mu,\nu$ are of the form (\ref{mu lightcone}). Corollary \ref{cor sphere lightcone} then gives the result.
\end{proof}

\subsection{Examples}
We now use the previous results to give new simple explicit examples of constant angle surfaces. Since examples of constant angle surfaces in lightcones were already given in Section \ref{section new example}, we will focus on examples in hyperspheres, and on examples which do not belong to hyperspheres nor to lightcones. All the examples arise from the application of Theorem \ref{Th::CoordintateForm} and its first corollary: we start with a given solution $\mu$ of \eqref{PDE-A}, we then find two functions $f,g$ solving \eqref{EQ::Systemfi}, and we finally obtain the immersion $F$ from the general form \eqref{F f ft g gt}. For the sake of brevity we only write the final expressions.
\subsubsection{Immersions in hyperspheres: $\mu= 2\sinh (c_1 y-c_2 x)$}
As proved in Corollary \ref{cor sphere lightcone}, the metric coefficient $\mu$ has to be of the form 
\[
\mu=  r_1 e^{c_1 y-c_2 x}+ r_2 e^{-(c_1 y-c_2 x)}.
\]
If we take $r_1=-r_2=1$ we get $\mu= 2\sinh (c_1 y-c_2 x)$ and then
$\nu=2\cosh (c_2 x+c_1 y)$. We assume that $(x,y)\in\R^2$ is such that $c_1 y\neq c_2 x,$ so that $\mu\neq 0.$
The immersion is thus
\[
\bfz=-\sinh (c_1 y-c_2 x)N_1+\cosh (c_1 y-c_2 x)N_2,
\]
whose norm is $|F|^2=-1$.

\subsubsection{Immersions which are not in hyperspheres nor in lightcones: $\mu= \sin (c_1 y+c_2 x)$} For this $\mu$, solution to \eqref{PDE-A}, we have $\nu=-\cos (c_2 x+c_1 y).$
Assuming that $\sin (c_1 y+c_2 x)\neq 0$ and  $\cos (c_2 x+c_1 y)\neq 0$ (for suitable $(x,y)$ in $\R^2$), the functions $f$ and $g$ are
\[
f(x,y)=\cos (c_2 x+c_1 y), \qquad g(x,y)=-\sin (c_2 x+c_1 y).
\]
Thus, the immersion reads
\[
\begin{array}{rl}
\bfz&=\cos (c_2 x+c_1 y)c_2 T_1 -\sin (c_2 x+c_1 y)c_1 T_2-\sin (c_2 x+c_1 y) N_1 -\cos (c_2 x+c_1 y) N_2
\\
&=-c_2 \nu T_1 -c_1 \mu T_2-\mu N_1 +\nu N_2,
\end{array}
\]
whose norm is
\[
|\bfz|^2=(c_1^2+1)\nu^2+(c_2^2-1)\mu^2=(c_1^2+1)\cos^2 (c_2 x+c_1 y)+(c_2^2-1)\sin^2 (c_2 x+c_1 y).
\]

\subsubsection{Another example: $\mu=(c_1 y - c_2 x)( \sin (c_1 y+c_2 x)+\cos (c_1 y+c_2 x))$}
For this $\mu$ we get
\[
\nu=(1-c_1 y + c_2 x)\cos (c_1 y+c_2 x)-(1+c_1 y - c_2 x) \sin (c_1 y+c_2 x) .
\] 
We assume that $(x,y)$ belongs to an open subset of $\R^2$ such that $\mu\nu \neq 0$. It appears that the functions
\[
\begin{array}{rl}
f(x,y)=&c_2 ((2+c_1^2-c_2^2)(c_1 y-c_2 x)(\cos(c_1 y+c_2 x)+\sin(c_1 y+c_2 x))
\\[8pt] &
+(c_1^2+c_2^2)(\cos(c_1 y+c_2 x)-\sin(c_1 y+c_2 x))) , \\[12pt] 
g(x,y)= &(2+c_1^2-c_2^2)(c_1 y-c_2 x)(\cos(c_1 y+c_2 x)-\sin(c_1 y+c_2 x))
\\[8pt] &
-(c_1^2+c_2^2)(\cos(c_1 y+c_2 x)+\sin(c_1 y+c_2 x))
\end{array}
\]
are solutions of \eqref{EQ::Systemfi}, which defines the immersion by \eqref{F f ft g gt}.

\appendix
\section{Angles and orthogonal projections}\label{app angle}

\subsection{An alternative construction of the complex angle}
We give here another construction of the complex angle between two oriented and spacelike planes $p$ and $q$ in $\R^{1,3}.$ Let us denote by $\pi:q\rightarrow p$ and $\pi':q\rightarrow p^{\perp}$ the restrictions to $q$ of the orthogonal projections $\R^{1,3}\rightarrow p$ and $\R^{1,3}\rightarrow {p}^{\perp},$ and consider the quadratic forms $Q$ and $Q'$ defined on $q$ by
$$Q(x)=\langle \pi(x),\pi(x)\rangle\hspace{.5cm}\mbox{and}\hspace{.5cm}Q'(x)=\langle \pi'(x),\pi'(x)\rangle$$
for all $x\in q.$ They are linked by the relation $Q(x)+Q'(x)=|x|^2$ for all $x\in q,$ and thus satisfy
\begin{equation}\label{q+qp}
Q(x)+Q'(x)=1
\end{equation}
for all $x\in q,$ $|x|=1.$ There exists a positive orthonormal basis $(u,u^{\perp})$ of $q$ such that 
$$Q(u)=\max_{x\in q,\ |x|=1}Q(x),\hspace{.5cm} Q(u^{\perp})=\min_{x\in q,\ |x|=1}Q(x).$$
By (\ref{q+qp}) this basis is also such that
$$Q'(u)=\min_{x\in q,\ |x|=1}Q'(x),\hspace{.5cm} Q'(u^{\perp})=\max_{x\in q,\ |x|=1}Q'(x),$$
and satisfies
$$\langle\pi(u),\pi(u^{\perp})\rangle=0\hspace{.5cm}\mbox{and}\hspace{.5cm}\langle\pi'(u),\pi'(u^{\perp})\rangle=0.$$
We need to divide the discussion into three main cases, according to the dimension of the range of $\pi'$:
\\
\\\textbf{Case 1: rank$(\pi')$=2.} This condition means that $p\oplus q=\R^{1,3}.$ Since $p^{\perp}$ is timelike, $Q'$ has then signature $(1,1).$ The relation (\ref{q+qp}) then implies that $\pi'(u)$ is timelike, $Q(u)>1$ and $Q(u^{\perp})<1$. We choose moreover $u$ such that $\pi'(u)$ is future-directed: the basis $(u,u^{\perp})$ is then uniquely defined. We then consider the positively oriented and orthonormal basis $(e_0,e_1,e_2,e_3)$ of $\R^{1,3}$ such that $e_0$ is future-directed, $(e_2,e_3)$ is a positive basis of $p$ and
$$\pi'(u)=a_0 e_0,\ \pi'(u^{\perp})=a_1 e_1,\ \pi(u)=a_2e_2,\ \pi(u^{\perp})=a_3e_3$$
for some constants $a_0,a_1,a_2,a_3$ such that $a_0,a_2\geq 0.$ Since
$$u=a_0 e_0+a_2e_2\hspace{.5cm}\mbox{and}\hspace{.5cm}u^{\perp}=a_1 e_1+a_3e_3$$
we have
$$|u|^2=1=-a_0^2+a_2^2\hspace{.5cm}\mbox{and}\hspace{.5cm}|u^{\perp}|^2=1=a_1^2+a_3^2$$
and we may set $\psi_1,\psi_2\in\R,$ $\psi_2\geq 0,$ such that
$$a_0=\sinh\psi_2,\ a_2=\cosh\psi_2,\ a_1=-\sin\psi_1\hspace{.3cm}\mbox{and}\hspace{.3cm}  a_3=\cos\psi_1.$$ 
Since $Q'(u)=-a_0^2,$ $Q'(u^{\perp})=a_1^2$ and the signature of $Q'$ is $(1,1),$ we in fact have $a_0a_1\neq 0,$ that is $\psi_2>0$ and $\psi_1\neq 0\ [\pi].$ With these definitions, we have
\begin{equation}\label{u u perp psi1 psi2}
u=\sinh\psi_2\ e_0+\cosh\psi_2\ e_2\hspace{.3cm}\mbox{and}\hspace{.3cm}u^{\perp}=-\sin\psi_1\ e_1+\cos\psi_1\ e_3
\end{equation}
which easily yields for $\psi:=\psi_1+i\psi_2$
$$q=u\wedge u^{\perp}=\cos\psi\ p+\sin\psi\ V$$
for $V=e_1\wedge e_2.$ 
\\
\\\textbf{Case 2: rank$(\pi')=1$.} In that case $Im(\pi')$ is a line, which may be timelike, spacelike or lightlike: 
\\$\bullet$ $Q'$ has signature $(0,1),$ i.e. $Im(\pi')$ is a timelike line: $q$ belongs to the timelike hyperplane $p\oplus Im(\pi').$ We may then follow the lines of the previous case: the basis $(u,u^{\perp})$ is such that $Q(u)>1$ and $Q(u^{\perp})=1,$ and we may suppose that $\pi'(u)$ is future-directed. We then define the basis $(e_0,e_1,e_2,e_3)$ of $\R^{1,3}$ and the angles $\psi_1$ and $\psi_2$ as above, and observe that $\psi_1=0\ [\pi]$ and $\psi_2>0$ ($Q'(u^{\perp})=a_1^2=0$ here since $Q(u^{\perp})=1$ and by (\ref{q+qp})). The angles $\psi_1$ and $\psi_2$ have the following interpretations: the planes $p$ and $q$ are two oriented planes in the timelike hyperplane $p\oplus Im(\pi');$ this hyperplane is naturally oriented by the orientation of $p$ and the future-orientation of $Im(\pi');$ the lines normal to the oriented planes $p$ and $q$ are thus also oriented, and $\psi_1=0\ [2\pi]$ if these lines are both future-oriented, and $\psi_1=\pi\ [2\pi]$ in the other case. The angle $\psi_2$ is the measure of the lorentzian angle between the future-directed lines normal to $p$ and $q.$
\\$\bullet$ $Q'$ has signature $(1,0),$ i.e. $Im(\pi')$ is a spacelike line: $q$ belongs to the spacelike hyperplane $p\oplus Im(\pi').$ The basis $(u,u^{\perp})$ is such that $Q(u)=1$ and $Q(u^{\perp})<1,$ i.e. $Q'(u)=0$ and $Q'(u^{\perp})>0.$ It is no more possible to suppose here that $\pi'(u)$ is timelike (it is the vector zero!), i.e. $u$ is only defined up to sign. Nevertheless, choosing for $e_0$ the future-directed unit vector orthogonal to $p\oplus\R\pi'(u^{\perp})$ we may construct as above a unique positively oriented and orthonormal basis $(e_0,e_1,e_2,e_3)$ of $\R^{1.3}$ adapted to $\pi'(u^{\perp}),$ $\pi(u)$ and $\pi(u^{\perp}).$ Defining $\psi_1$ and $\psi_2$ as above we easily see that $\psi_2=0$ and $\psi_1$ is defined up to sign and a multiple of $2\pi:$ it is thus represented by a unique real number in $[0,\pi].$ This angle has the following interpretation: the planes $p$ and $q$ are two oriented planes in the spacelike hyperplane $p\oplus Im(\pi')$ which is naturally oriented by the election of its future-directed normal direction in $\R^{1,3}$ and the canonical orientation of $\R^{1,3};$ $\psi_1\in [0,\pi]$ is a measure of the angle of the oriented lines normal to these planes in $p\oplus Im(\pi').$
\\$\bullet$ $Q'$ is zero, i.e. $Im(\pi')$ is a lightlike line: $q$ belongs to the degenerate hyperplane $p\oplus Im(\pi').$ In that case we set $\psi_1=0\ [\pi]$ and $\psi_2=0$: more precisely, a given orientation of $Im(\pi')$ induces an orientation of $p\oplus Im(\pi'),$ and we set $\psi_1=0\ [2\pi]$ if it coincides with the sum of the orientations of $q$ and $Im(\pi'),$ and  $\psi_1=\pi\ [2\pi]$ in the other case.
\\
\\\textbf{Case 3: rank$(\pi')$=0.} In that case $q=\pm p,$ i.e. the planes $p$ and $q$ coincide in $\R^{1,3},$ with the same orientation ($q=p$) or the opposite orientation ($q=-p$); we set $\psi_2=0$ and $\psi_1=0\ [2\pi]$ if $q=p$ or $\psi_1=\pi\ [2\pi]$ if $q=-p.$
\\
\\It is not difficult to verify that in all the cases described above the complex angle $\psi:=\psi_1+i\psi_2$ coincides with the complex angle $\psi$ defined in Section \ref{section def angle}, if we moreover choose the sign of $\psi$ in Definition \ref{def complex angle} such that $\psi_2:=\Im m\ \psi\geq 0.$ We have obtained the following additional informations concerning the relative position of $q$ and $p$ in terms of the angle $\psi=\psi_1+i\psi_2$: 
\begin{prop}\label{prop pos p p0}
\begin{itemize}
\item $q$ does not belong to any hyperplane containing $p$ if and only if $\psi_1\neq 0\ [\pi]$ and $\psi_2>0;$ 
\item $q\neq \pm p$ belongs to a spacelike (resp. timelike) hyperplane containing $p$ if and only if $\psi_1\neq 0\ [\pi]$ and $\psi_2=0$ (resp. $\psi_1=0\ [\pi]$ and $\psi_2>0$);
\item $q$ belongs to a null hyperplane containing $p$ if and only if $\psi_1=0\ [\pi]$ and $\psi_2=0.$
\end{itemize}
\end{prop}
\subsection{An elementary characterization of the two real angles}
We still suppose that $p$ and $q$ are oriented spacelike planes in $\R^{1,3}$ and we consider their real angles $\psi_1,\psi_2$ constructed above.The following descriptions of $\psi_1$ and $\psi_2$ are very similar to the elementary definitions of the two principal angles between two oriented planes in $\R^4$ (see for example \cite{BdSOR}).
\begin{prop} We have
$$\cosh\psi_2=\sup_{x\in q,\ y\in p,\ |x|=|y|=1}\langle x,y\rangle.$$
Moreover, if $u\in q$ and $v\in p,$ $|u|=|v|=1$ are such that $\cosh\psi_2=\langle u,v\rangle,$ then
$$\cos\psi_1=\langle u^{\perp},v^{\perp}\rangle$$
where $u^{\perp}$ and $v^{\perp}$ are unit vectors in $q$ and $p$ such that $(u,u^{\perp})$ and $(v,v^{\perp})$ are positive orthonormal bases of $q$ and $p.$
\end{prop}
\begin{proof}
We assume for simplicity that rank($\pi'$)=2 and keep the notations introduced in the previous section to study that case. If $x\in q$ and $y\in p$ are unit vectors, we have
$$\langle x, y\rangle=\langle \pi(x),y\rangle\leq |\pi(x)|\leq \sup_{x\in q,\ |x|=1}|\pi(x)|=\cosh\psi_2.$$
Moreover, the equality holds for $x=u=\sinh\psi_2\ e_0+\cosh\psi_2\ e_2$ and $y=v=e_2$ where $(e_0,e_1,e_2,e_3)$ is the basis defined above (Case 1). This proves the first claim. Since $u^{\perp}=-\sin\psi_1\ e_1+\cos\psi_1\ e_3$ (by (\ref{u u perp psi1 psi2})) and $v^{\perp}=e_3$ the last claim readily follows.
\end{proof}
\section{Characterization of surfaces of constant angle $\psi=0\ [\pi]$}\label{app case psi 0 mod pi}
We study here the very special case of the surfaces with regular Gauss map and constant angle $\psi=0\ [\pi]$ with respect to some given spacelike plane $p_o.$ The result is the following:
\begin{prop}
A connected and spacelike surface with regular Gauss map has a constant angle $\psi=0\ [\pi]$ with respect to a fixed spacelike plane $p_o$ if and only if it belongs to a affine hyperplane $x_o+p_o\oplus L$ where $x_o$ is a point of $\R^{1,3}$ and $L$ is one of the two null lines of the timelike plane $p_o^{\perp}.$
\end{prop}
\begin{proof}
We only prove the "only if" direction (the converse is trivial). Let us fix a positive orthonormal basis $(u,u^{\perp})$ of $p_o$ and a positive basis $(N,N')$ of $p_o^{\perp}$ such that 
$$\langle N,N\rangle=\langle N',N'\rangle=0\hspace{.5cm} \mbox{and} \hspace{.5cm}\langle N,N'\rangle=-1,$$
and set 
$$\xi=u\wedge N \hspace{.5cm} \mbox{and} \hspace{.5cm}\xi'=u^{\perp}\wedge N'.$$
$(\xi,\xi')$ form a basis of the $\C$-linear space $T_{p_o}\mathcal{Q}$ such that
$$H(\xi,\xi)=H(\xi',\xi')=0 \hspace{.5cm} \mbox{and} \hspace{.5cm}H(\xi,\xi')=i.$$
Recall the definition of the complex parameter $z$ at the end of Section \ref{subsection gauss map}. Since $H(G(z),p_o)=\cos\psi=1$ we have 
$$H(G(z)-p_o,p_o)=H(G(z),p_o)-H(p_o,p_o)=1-1=0$$
and
\begin{eqnarray*}
H(G(z)-p_o,G(z)-p_o)&=&H(G(z),G(z))+H(p_o,p_o)-2H(G(z),p_o)\\
&=&1+1-2\\
&=&0,
\end{eqnarray*}
which means that $G(z)-p_o$ belongs to one of the two complex lines $\C\xi$ and $\C\xi'$ of $T_{p_o}\mathcal{Q}.$ Thus $G'(z)$ also belongs $\C\xi\cup \C\xi'$ and by continuity $G'(z)$ belongs to $\C\xi$ for all value of $z,$ or $G'(z)$ belongs to $\C\xi'$ for all value of $z$ ($G'$ does not vanish since $G$ is assumed to be regular!). If $G'$ belongs to $\C\xi$ then $G$ belongs to $p_o+\C\xi,$ which means that all the tangent planes of $M$ belongs the hyperplane $p_o\oplus L$ with $L=\R N.$ This finally easily implies that $M$ belongs to an affine hyperplane $x_o+ p_o\oplus L$. 
\end{proof}

\section{Details for the computations of the adapted frame}\label{App::CompFrame} 

Let us briefly describe here the computations leading to the explicit formulas for the special frame $(T_1,T_2,N_1,N_2)$ adapted to a complex angle surface. Recalling Theorem \ref{thm representation 2}, we have
\[
T_1:=\frac{1}{\mu}\partial_x F  = \frac{1}{\mu}\xi(\partial_x)=\frac{1}{\mu}g^{-1} (\omega_1(\partial_x)J+\omega_2(\partial_x)K)\ \widehat{g} =g^{-1} (\sin(u) J-\cos(u)K)\ \widehat{g},
\]
where $g,\widehat{g}$ and $u$ are given by the formulas
\begin{eqnarray*}
g^{-1}&= \bar{g}=&-\cos\frac{\psi}{2}\ \sin\left(z\tan\frac{\psi}{2}\right)1-\cos\frac{\psi}{2}\ \cos\left(z\tan\frac{\psi}{2}\right)I\\
&&-\sin\frac{\psi}{2}\ \cos\left(z\cot\frac{\psi}{2}\right)J+\sin\frac{\psi}{2}\ \sin\left(z\cot\frac{\psi}{2}\right)K,
\end{eqnarray*}
\begin{eqnarray*}
\widehat{g}&=&-\cos\frac{\bar{\psi}}{2}\ \sin\left(\bar{z}\tan\frac{\bar{\psi}}{2}\right)1+\cos\frac{\bar{\psi}}{2}\ \cos\left(\bar{z}\tan\frac{\bar{\psi}}{2}\right)I\\
&&+\sin\frac{\bar{\psi}}{2}\ \cos\left(\bar{z}\cot\frac{\bar{\psi}}{2}\right)J-\sin\frac{\bar{\psi}}{2}\ \sin\left(\bar{z}\cot\frac{\bar{\psi}}{2}\right)K
\end{eqnarray*}
and 
$$u=-2\re(z \cot(\psi)).$$ 
Direct and long computations (or the use of a software of symbolic computation) lead to the formula
\[
\begin{array}{rl}
T_1=& 
-\sin \left(\dfrac{\psi - \bar{\psi}}{2} \right) 
\cos \left(z \csc \psi-\bar{z}\csc \bar{\psi}\right)
 1+ \sin \left(\dfrac{\psi - \bar{\psi}}{2} \right)  \sin \left(z\csc \psi -\bar{z}\csc \bar{\psi}\right)I
\\[8pt] & 
+  \cos\left(\dfrac{\psi - \bar{\psi}}{2} \right)  \sin \left(z\csc\psi+\bar{z}\csc\bar{\psi}\right)
 J+ \cos\left(\dfrac{\psi - \bar{\psi}}{2} \right)\cos \left(z\csc\psi+\bar{z} \csc\bar{\psi}\right)K
\end{array}
\]
where $\csc (\psi)=1/\sin\psi.$ Writing $\psi=\psi_1+i\psi_2$ and $z/\sin(\psi)=\varphi_1+i\varphi_2,$ we finally obtain
$$T_1=-i\sinh \psi_2 \cosh  \varphi_2 1- \sinh \psi_2  \sinh \varphi_2 I+\cosh \psi_2  \sin \varphi_1 J+\cosh\psi_2 \cos\varphi_1 K$$
which is the expression given in Section \ref{section explicit expression frame}. Similar computations give the expression of $T_2.$ We now explain how to determine $N_1$ and $N_2.$ On one side, we have by (\ref{ff Ti}) and (\ref{dG w})
$$dG(T_1)=\mathbf{II}(T_1,T_1)\wedge T_2=\frac{2}{\mu}\ N_1\wedge T_2$$
and
$$dG(T_2)=T_1\wedge \mathbf{II}(T_2,T_2)=\frac{2}{\nu}\ T_1\wedge N_2.$$
On the other side,  since $dG=G'dz$ we have
$$dG(T_1)=\frac{1}{\mu}dG(\partial_x)=\frac{1}{\mu}G'=\frac{2}{\mu}J\{-\sin\varphi+\cos\varphi I\}$$
and
$$dG(T_2)=\frac{1}{\nu}dG(\partial_y)=\frac{i}{\nu}G'=\frac{2i}{\nu}J\{-\sin\varphi+\cos\varphi I\}.$$
We deduce that, in $\HC,$
$$N_1=-J\{-\sin\varphi+\cos\varphi I\}T_2$$
and
$$N_2=iT_1J\{-\sin\overline{\varphi}+\cos\overline{\varphi} I\}.$$
The expressions of $N_1$ and $N_2$ then follow from the expressions of $T_1$ and $T_2,$ and direct computations.
\\
\\
\noindent\textbf{Acknowledgments:} P. Bayard was supported by the project PAPIIT-UNAM IA106218, J. Monterde was partially supported by the Spanish Ministry of Economy and Competitiveness DGICYT grant MTM2015-64013 and R. C. Volpe was supported by the Generalitat Valenciana VALi+D grant ACIF/2016/342.


\begin{thebibliography}{}
\bibitem{Bay} P. Bayard, \emph{On the spinorial representation of spacelike surfaces into 4-dimensional Minkowski space}, J. Geom. Phys. \textbf{74} (2013) 289-313.
\bibitem{BaySB} P. Bayard and F. S\'anchez-Bringas, \emph{Geometric invariants and principal configurations on spacelike surfaces immersed in $\R^{3,1}$}, Proc. Roy. Soc. Edinburgh: Sec. A Math. \textbf{140} (2010) 1141-1160.
\bibitem{BdSOR} P. Bayard, A. Di Scala, O. Osuna Castro and G. Ruiz Hern\'andez, \emph{Surfaces in $\R^4$ with constant principal angles with respect to a plane}, Geom. Dedicata. \textbf{162} (2013) 153-176.
\bibitem{DRH} A. Di Scala and  G. Ruiz-Hern\'andez, \emph{Helix submanifolds of euclidean spaces}, Monatsh Math. \textbf{157} (2009) 205-215.
\bibitem{Fr} T. Friedrich, \emph{Dirac Operators in Riemannian Geometry}, Graduate Studies in Mathematics \textbf{25} AMS (2000).
\bibitem{GMM} J.A. G\'alvez, A. Mart\'{\i}nez and F. Mil\'an, \emph{Flat surfaces in $\mathbb{L}^4$}, Ann. Glob. Anal. Geom. \textbf{20:3} (2001) 243-251.
\bibitem{Kosh} N. S. Koshyakov, M. M. Smirnov and E. B. Gliner, \emph{Differential equations of  mathematical physics}, North-Holland Publishing Company (1964) 31-41.
\bibitem{LM} R. L\'opez and M. I. Munteanu, \emph{Constant angle surfaces in Minkowski space}, Bull. Belg. Math. Soc. Simon Stevin \textbf{18} (2011) 271-286.
\bibitem{MV} J. Monterde and R. Volpe, \emph{Explicit immersions of surfaces in $\R^4$ with arbitrary constant Jordan angles}, preprint.
\bibitem{RH} G. Ruiz-Hern\'andez, \emph{Helix, shadow boundary and minimal submanifolds}, Illinois J. of Math. \textbf{52:4} (2008) 1385-1397.
\end{thebibliography}
\end{document}